\documentclass{amsart}
\usepackage{amssymb}
\usepackage[all]{xy}
\usepackage{graphicx}
\usepackage[mathcal]{euscript}
\usepackage{verbatim}

\let\oldmarginpar\marginpar
\renewcommand\marginpar[1]{\-\oldmarginpar[\raggedleft\footnotesize #1]%
{\raggedright\footnotesize #1}}

\usepackage[hang]{caption} 
\usepackage{mathtools}

\usepackage{color}

\newtheorem{thm}{Theorem}[section]
\newtheorem{lem}[thm]{Lemma}
\newtheorem{prop}[thm]{Proposition}
\newtheorem{cor}[thm]{Corollary}

\theoremstyle{definition}
\newtheorem{defn}[thm]{Definition}
\newtheorem*{conv}{Convention}
\theoremstyle{remark}
\newtheorem{rmk}[thm]{Remark}
\newtheorem{exam}[thm]{Example}

\newcommand{\tz}{\tilde{\zeta}}

\newcommand{\nt}{\nabla_\tau}
\newcommand{\C}{\mathbb{C}}
\newcommand{\N}{\mathbb{N}}
\newcommand{\Z}{\mathbb{Z}}
\newcommand{\R}{\mathbb{R}}
\newcommand{\Q}{\mathbb{Q}}
\newcommand{\g}{\gamma}
\renewcommand{\S}{\Sigma}

\newcommand{\s}{\sigma}
\newcommand{\de}{\delta}
\newcommand{\Ad}{\mathrm{Ad}}
\newcommand{\ad}{\mathrm{ad}}
\renewcommand{\a}{\alpha}
\newcommand{\Om}{\Omega}
\newcommand{\fo}{\mathfrak{o}}
\newcommand{\fu}{\mathfrak{u}}

\newcommand{\hfu}{\hat{\fu}}

\newcommand{\vr}{\varrho}
\newcommand{\hvr}{\hat{\varrho}}
\renewcommand{\b}{\beta}

\newcommand{\tb}{\tilde{\beta}}
\newcommand{\tbo}{\tilde{\beta}_0}

\newcommand{\brho}{\bar{\rho}}

\newcommand{\triv}{\mathrm{triv}}

\newcommand{\fb}{\mathfrak{b}}
\newcommand{\fz}{\mathfrak{z}}
\newcommand{\ddz}{\frac{dz}{z}}
\newcommand{\n}{\nabla}
\newcommand{\fg}{\mathfrak{g}}

\newcommand{\fh}{\mathfrak{h}}
\newcommand{\hG}{\hat{G}}
\newcommand{\hfg}{\hat{\fg}}

\newcommand{\hW}{\hat{W}}
\newcommand{\hT}{\hat{T}}

\newcommand{\hV}{\hat{V}}
\newcommand{\hft}{\hat{\ft}}
\newcommand{\ft}{\mathfrak{t}}
\newcommand{\fs}{\mathfrak{s}}

\newcommand{\tx}{\tilde{x}}
\newcommand{\ty}{\tilde{y}}
\newcommand{\tX}{\tilde{X}}
\newcommand{\tA}{\tilde{A}}
\newcommand{\hN}{\hat{N}}
\newcommand{\tw}{\tilde{w}}

\newcommand{\sG}{\mathcal{G}}
\newcommand{\nbr}{[\nabla]}

\newcommand{\cF}{\mathcal{F}}

\newcommand{\A}{\mathcal{A}}

\newcommand{\B}{\mathcal{B}}
\newcommand{\bB}{\bar{\B}}

\newcommand{\Ao}{\A_0}

\newcommand{\bF}{\bar{F}}

\newcommand{\BT}{\mathcal{B}}

\DeclareMathOperator{\GL}{\mathrm{GL}}

\DeclareMathOperator{\AT}{\mathbf{A}}

\DeclareMathOperator{\Hom}{\mathrm{Hom}}
\DeclareMathOperator{\Aff}{\mathrm{Aff}}
\DeclareMathOperator{\Res}{\mathrm{Res}}

\DeclareMathOperator{\Gal}{\mathrm{Gal}}
\DeclareMathOperator{\Lie}{\mathrm{Lie}}

\DeclareMathOperator{\Rep}{\mathrm{Rep}}
\DeclareMathOperator{\Spec}{\mathrm{Spec}}

\DeclareMathOperator{\tr}{\mathrm{tr}}
\DeclareMathOperator{\cC}{\mathcal{C}}
\DeclareMathOperator{\id}{\mathrm{id}}
\newcommand{\cCfr}{\cC^{\mathrm{fr}}}

\newcommand{\Dx}{\Delta^\times}

\newcommand{\hw}{\hat{w}}
\newcommand{\bG}{\bar{\Gamma}}

\DeclareMathOperator{\gl}{\mathfrak{gl}}

\newcommand{\bfy}{\mathbf{y}}

\newcommand{\bfS}{\mathbf{S}}
\newcommand{\bfr}{\mathbf{r}}
\newcommand{\bfA}{\mathbf{A}}

\DeclareMathOperator{\tM}{\widetilde{\mathcal{M}}}

\DeclareMathOperator{\cM}{\mathcal{M}}
\DeclareMathOperator{\cA}{\mathcal{A}}

\newcommand{\hn}{\hat{\n}}
\DeclareMathOperator{\slope}{slope}

\DeclareMathOperator{\Crit}{Crit}
\renewcommand{\P}{\mathbb P}
\newcommand{\fP}{\mathfrak{P}}

\title[Flat $G$-bundles and
  regular strata for reductive groups]{Flat $G$-bundles and
  regular strata for reductive groups}

\author{Christopher L.~Bremer}
\address{Department of Mathematics\\
  Louisiana State University\\
  Baton Rouge, LA 70803} 
\email{cbremer@math.lsu.edu} 
\author{Daniel
  S.~Sage} 
\email{sage@math.lsu.edu}

\email{}
\thanks{The research of the
  second author was partially supported by a grant from the Simons
  Foundation (\#281502).}
\subjclass[2010]{}
\keywords{}

\begin{document}

\begin{abstract}
  Let $\hG$ be an algebraic loop group associated to a reductive group
  $G$.  A fundamental stratum is a triple consisting of a point $x$ in
  the Bruhat-Tits building of $\hG$, a nonnegative real number $r$,
  and a character of the corresponding depth $r$ Moy-Prasad subgroup
  that satisfies a non-degeneracy condition.  The authors have shown
  in previous work how to associate a fundamental stratum to a formal
  flat $G$-bundle and used this theory to define its slope.  In this
  paper, the authors study fundamental strata that satisfy an
  additional regular semisimplicity condition.  Flat $G$-bundles that
  contain regular strata have a natural reduction of structure to a
  (not necessarily split) maximal torus in $\hG$, and the authors use
  this property to compute the corresponding moduli spaces.  This
  theory generalizes a natural condition on algebraic connections (the
  $\GL_n$ case), which plays an important role in the global analysis
  of meromorphic connections and isomonodromic deformations.
\end{abstract}

\maketitle
\section{Introduction}

The study of meromorphic connections on algebraic curves (or
equivalently, flat $\GL_n(\C)$-bundles) often reduces to the analysis
of the associated formal connections at each pole.  This
local-to-global approach has proven to be especially effective when
the principal part of the connection at any irregular singular point
has a diagonalizable leading term with distinct eigenvalues.  For
example, the first significant progress on the isomonodromy problem
for irregular singular differential equations came in a 1981 paper of
Jimbo, Miwa, and Ueno, in which they imposed this condition at the
singularities~\cite{JMU}.  Also in this context, Boalch has
constructed well-behaved moduli spaces of connections on $\P^1$ with
given formal isomorphism classes at the singularities and has further
exhibited the isomonodromy equations as an integrable system on an
appropriate Poisson manifold~\cite{Boa}.  Analogous results hold for
flat $G$-bundles, where $G$ is a complex reductive group~\cite{Fed}.
Other aspects of the monodromy map for flat $G$-bundles of this type
have been studied in \cite{BrLa, Boa4}.

While there is a very satisfactory picture of this type of connection,
the conditions imposed are quite restrictive.  Indeed, such
connections necessarily have integral slope at each singularity
whereas the slope of a rank $n$ formal connection at an irregular
singular point can be any positive rational number with denominator at
most $n$.  Moreover, many connections of particular interest are not
of this type.  Recall that in the $\GL_n$ case of the geometric
Langlands program, the role of Galois representations is played by
monodromy data associated to flat connections: over a smooth complex
curve $X$ or the formal punctured disk $\Dx=\Spec(F)$ depending on
whether one is in the global or local context.  By analogy with the
classical situation, one expects that connections corresponding to
cuspidal representations will not have regular semisimple leading
terms.  For example, Frenkel and Gross have constructed a rigid flat
$G$-bundle (for any reductive $G$) which corresponds to the Steinberg
representation at $0$ and a certain ``small'' supercuspidal
representation at $\infty$~\cite{FrGr}.  When $G=GL_2(\C)$, this is
just the classical Airy connection.  Here, the leading term at the
irregular singular point at $\infty$ is nilpotent.

In \cite{BrSa1}, the authors generalized Boalch's results mentioned
above to a much wider class of meromorphic connections.  This was done
through the introduction of a new notion of the ``leading term'' of a
formal connection in terms of a geometric version of the theory of
fundamental strata familiar from $p$-adic representation theory (see,
for example, ~\cite{BuKu1}).  Let $F=\C((z))$ be the field of formal
Laurent series with ring of integers $\fo=\C[[z]]$.  A
$\GL_n(F)$-stratum is a triple $(P,r,\b)$ with $P\subset\GL_n(F)$ a
parahoric subgroup, $r$ a nonnegative integer, and $\b$ a functional
on the quotient of congruent subalgebras $\fP^r/\fP^{r+1}$.  The
stratum is fundamental if $\b$ satisfies a certain nondegeneracy
condition.  Let $(\hV,\hn)$ be a rank $n$ connection over the formal
punctured disk $\Dx=\Spec(F)$.  After fixing a trivialization for
$\hV$, the matrix of the connection $[\hn]$ is an element of
$\gl_n(F)\frac{dz}{z}$.  In particular, it induces a functional on
$\gl_n(F)$ via taking the residue of the trace form.  We say that
$(\hV,\hn)$ contains the stratum $(P,r,\b)$ if this functional kills
$\fP^{r+1}$ and induces $\b$ on the quotient space.  Every connection
contains a fundamental strata, and each such stratum should be viewed
as a ``correct'' leading term of the connection.  For example, a
stratum determines the slope of an irregular connection if and only if
it is fundamental.  The Frenkel-Gross connection does not contain a
fundamental stratum at the irregular singular point with respect to
the usual filtration (with $P=GL_n(\fo)$), but it does with respect to
a certain Iwahori subgroup.

The key property that allows one to construct smooth moduli spaces of
global connections is for the corresponding formal connections to
contain \emph{regular strata}.  These are fundamental strata which are
centralized in a graded sense by a possibly nonsplit maximal torus
$S\subset\GL_n(F)$.  For example, if $[\n]=(M_{-r}z^{-r}+
M_{-r+1}z^{-r+1}+\dots)\frac{dz}{z}$ with $M_i\in\gl_n(\C)$ and
$M_{-r}$ regular semisimple, then it contains a regular stratum
$(\GL_n(\fo),r,\b)$ centralized by the diagonal torus.  Only certain
conjugacy classes of maximal tori can centralize a regular stratum,
and only \emph{uniform} maximal tori--maximal tori which are the
product of some number of copies of $E^\times$ for some field
extension $E$ of $F$--are considered in~\cite{BrSa1,BrSa2}.  The
Frenkel-Gross connection for $\GL_n(\C)$ contains a regular stratum
centralized by a maximal torus isomorphic to $F[z^{1/n}]^\times$. If
$(\hV,\hn)$ contains a regular stratum $(P,r,\b)$ centralized by $S$,
then we show that its matrix is gauge-equivalent to an element of $\fs
\frac{dz}{z}$, where $\fs=\Lie(S)$.  In fact, we construct a certain
affine subvariety of $\fs_{-r}/\fs_{-1}$ called the variety of
$S$-formal types of depth $r$, which admits a free action of $\hW_S$,
the relative affine Weyl group of $S$.  (Here, $\fs_k$ is the $k$-th
piece of the natural filtration on $\fs$.)  We then show that the
moduli space of such connections is isomorphic to the set of
$\hW_S$-orbits.

In \cite{BrSa1,BrSa2}, we generalize Boalch's results to meromorphic
connections on $\P^1$ which contain regular strata at each irregular
singular point.  In particular, consider meromorphic connections
$(V,\n)$ with singularities at $\bfy=(y_1,\dots,y_m)$ and which
contain regular strata $(P_i,r_i,\b_i)$ centralized by $S_i$ at each
$y_i$.  We then construct a Poisson manifold $\tM(\bfy,\bfS,\bfr)$ of
such connections with given ``framing data''.  If $\bfA$ is an
$m$-tuple of formal types with the combinatorics determined by $\bfS$
and $\bfr$, we also construct the space $\cM(\bfy,\bfA)$ (resp.
$\tM(\bfy,\bfA)$) of framable (resp. framed) connections with the
specified formal types.  The variety $\cM(\bfy,\bfA)$ is the
symplectic reduction of the symplectic manifold $\tM(\bfy,\bfA)$ via a
torus action.  The constructions of all of these spaces are
automorphic, in the sense that they are realized as the symplectic or
Poisson reduction of products of smooth varieties determined by local
data.  Finally, the monodromy map and the formal types map induce
orthogonal foliations on $\tM(\bfy,\bfS,\bfr)$.  Thus, the fibers of
the monodromy map are the leaves of an integrable system on
$\tM(\bfy,\bfS,\bfr)$ determined by the isomondromy equations while
the connected components of the $\tM(\bfy,\bfA)$'s are the symplectic
leaves of $\tM(\bfy,\bfS,\bfr)$.

The goal of this paper is to develop the local theory necessary to
obtain similar results for flat $G$-bundles.  In particular, we
generalize the theory of regular strata and its application to formal
$G$-bundles.  Our starting point is the geometric theory of
fundamental strata for reductive groups~\cite{BrSamink}, which we
review in Section~\ref{sec:prelim}.  Given any point $x$ in the
Bruhat-Tits building $\B$ for $G(F)$, Moy and Prasad have defined a
decreasing $\R$-filtration $\left(\fg_{x,r}\right)$ on $\fg(F)$ with a
discrete number of steps~\cite{MP1,MP2}.  A stratum is a triple
$(x,r,\b)$ where $x\in\B$, $r\in\R_{\ge 0}$, and $\b$ is a functional
on the $r$-th step $\fg_{x,r}/\fg_{x,r+}$ in the filtration.  In
\cite{BrSamink}, we show that every flat $G$-bundle contains a
fundamental stratum and the stratum depth $r$ is the same for all
of them.  We thus obtain a new invariant for formal flat $G$-bundles
called the slope.  These results are the geometric analogue of Moy and
Prasad's theorem on the existence of minimal $K$-types for admissible
representations of $p$-adic groups~\cite{MP1,MP2}.

Intuitively, regular strata are fundamental strata that satisfy a
graded version of regular semisimplicity. Regular strata do not appear
in the $p$-adic theory, though they have some points in common with
the semisimple strata considered for $p$-adic classical groups in
\cite{BuKu99, St}.  As a preliminary, we first study points in the building
compatible with a given Cartan subalgebra.  A point $x$ is compatible
with the Cartan subalgebra $\fs$ if the restriction of the filtration
given by $x$ to $\fs$ is the unique Moy-Prasad filtration on $\fs$.
If $\Ao\subset\B$ is a fixed rational apartment, Theorems~\ref{torcon}
and \ref{gx} give existence and classification results for Cartan
subalgebras graded compatible with a given point in $\Ao$.  We apply
these results in Corollary~\ref{strongcompat} to classify the set of
points in $\Ao$ compatible with some conjugate of $\fs$.

In the following section, we introduce the concept of an $S$-regular
stratum $(x,r,\b)$, where $S$ is a maximal torus in $G(F)$.  Roughly
speaking, this means that $x$ is compatible with the associated Cartan
subalgebra $\fs$ and that every representative of $\b$ has connected
centralizer a suitable conjugate of $S$.  The existence of an
$S$-regular stratum is a restrictive condition.  Recalling that the
classes of maximal tori in $\hG$ correspond bijectively to the
conjugacy classes in the Weyl group $W$, we show in
Corollary~\ref{regtorus} that it can only occur when $S$ corresponds
to a regular conjugacy class in $W$.  For example, when $G=\GL_n$,
such maximal tori are the uniform maximal tori and tori of the form
$S'\times F^\times$ where $S'$ is uniform in $\GL_{n-1}(\C)$.
Combining this with Corollary~\ref{strongcompat}, we obtain a
description of all points in $\Ao$ which can support a regular stratum
for a given conjugacy class of maximal tori.

Finally, in Section~\ref{isosect}, we study the category $\cC(S,r)$ of
formal flat $G$-bundles which contain an $S$-regular stratum of slope
$r$ and an associated category $\cCfr_x(S,r)$, depending on a choice
of compatible point $x$, of framed flat bundles.  (When $S$ is split,
we take $S=T(F)$ and only allow $x$ to be the vertex corresponding to
$G(\fo)$.)  We show that the framed categories are independent of the
choice of $x$. The moduli space of $\cCfr_x(S,r)$ can be viewed as
the set $\AT(S,r)$ of $S$-formal types of depth $r$--a certain affine
open (when $r>0$) subset of $\fs^\vee_{-r}/\fs^\vee_{0+}$ endowed with
a free action of the relative affine Weyl group $\hW_S$.
Theorem~\ref{diagthm} 
states that any $\hn$ containing an $S$-regular stratum is
gauge-equivalent to a flat $G$-bundle determined by a formal type in
$\AT(S,r)$.  More precisely, the forgetful deframing functor
$\cCfr_x(S,r)\to \cC(S,r)$ induces the quotient map
$\AT(S,r)\to\AT(S,r)/\hW_S$ on moduli spaces.

We expect that the results on meromorphic connections in
\cite{BrSa1,BrSa2} can be generalized to meromorphic flat $G$-bundles
containing regular strata at each irregular singular point.  We are
also hopeful that these results will be of use in the geometric
Langlands program.  In particular, we anticipate that there is an
interpretation of fundamental strata for representations of affine
Kac-Moody algebras and that representations containing regular strata
should correspond to formal flat $G$-bundles containing regular
strata.

\section{Preliminaries}\label{sec:prelim}

Let $k$ be an algebraically closed field of characteristic $0$, and
let $G$ be a connected reductive group over $k$ with Lie algebra
$\fg$.  Fix a maximal torus $T\subset G$ with corresponding Cartan
subalgebra $\ft$.  Let $N=N(T)$ be the normalizer of $T$, so that the
Weyl group $W$ of $G$ is isomorphic to $N/T$.  The set of roots with
respect to $T$ will be denoted by $\Phi$.  Given $\alpha \in \Phi$,
$U_\a \subset G$ is the associated root subgroup and $\fu_\a \subset
\fg$ is the weight space for $\ft$ corresponding to $\a$.  We will
write $Z$ for the center of $G$ and $\fz$ for its Lie algebra.  We fix
a nondegenerate invariant symmetric bilinear form $\left<, \right>$ on
$\fg$ throughout.  Finally, $\Rep(G)$ denotes the category of
finite-dimensional representations of $G$ over $k$.

Let $F=k((z))$ be the field of formal Laurent series over $k$ with
ring of integers $\fo=k[[z]]$, and let $\Dx = \Spec (F)$ be the formal
punctured disk.  We denote the Euler differential operator on $F$ by
$\tau=z \frac{d}{d z}$. We set $\hG=G(F)$ and $\hfg=\fg\otimes_k F$;
note that $\hG$ represents the functor sending a $k$-algebra $R$ to
$G(R((z)))$.  We will use the analogous notation $\hat{H}$ and
$\hat{\fh}$ for any algebraic group $H$ over $k$.  Similarly, if $V$
is a representation of $G$, then $\hV=V\otimes F$ will denote the
corresponding representation of $\hG$.

The Bruhat-Tits building and the enlarged building of $\hG$ will be
denoted by $\bB$ and $\B$ respectively.  If $x\in\B$, we denote the
corresponding parahoric subgroup (resp. subalgebra) by $\hG_x$
(resp. $\hfg_x$).  The standard apartment in
$\B$ associated to the split rational torus $\hT=T(F)$ is an affine
space isomorphic to $X_*(T) \otimes_\Z \R$.  If $\R\subset k$, then
points in $\Ao$ may be viewed as elements of $\ft_\R$.  The map
$\Ao\to\ft_\R$ is induced by evaluating cocharacters at $1$.  We write
$\tx\in\ft_\R$ for the image of $x\in\Ao$.  If $x\in X_*(T) \otimes_\Z
\Q$, then $\tx\in\ft$ is defined for any $k$.  

\begin{conv}\label{conv}  If the notation $\tx$ is used for $x\in\Ao$,
  then either $k$ contains $\R$ or $x$ is a rational point of $\Ao$ (so $\tx\in\ft_\Q$).
\end{conv}

Let $\bar{F}$ be an algebraic closure of $F$.  Later in the paper, we
will need to consider elements of $T(\bar{F})$ of the form $z^v$ with
$v\in\ft_\Q$.  Recall that $\bar{F}$ is generated by $m$-th roots of
$z$.  Suppose that $u\in\bar{F}$ satisfies $u^m=z$.  If
$v\in\ft_{\frac{1}{m}\Z}$, one can define $z^v$ as the unique element
of $T(F[u])$ satisfying $\chi(z^v)=u^{d\chi(mv)}$ for all $\chi\in
X^*(T)$.  This, of course, depends on the choice of $m$ and $u$, but
it is well-defined up to multiplication by $\xi^v\in T$, where $\xi$
is an $m^{th}$ root of unity.  Since $\xi^v\in T$, and in particular
is fixed by $\Gal(\bar{F}/F)$ and killed by $\tau$, it will follow
that all results that involve $z^v$ will be independent of this
choice.  For convenience, we may assume that all elements of this form
are defined in terms of a coherent set of uniformizers for the finite
extensions of $\bar{F}$, i.e., a choice of elements $u_m\in\bar{F}$
for each $m\in\N$ satisfying $u_m^m=z$ and such that if $m'|m$, then
$u_m^{m/{m'}}=u_{m'}$.

\subsection{Moy-Prasad filtrations}

If $V$ is any finite-dimensional representation of $G$, then any point
$x\in\B$ induces a decreasing $\R$-filtration $\{\hV_{x,r}\}$ of $\hV$
by $\fo$-lattices called the Moy-Prasad filtration~\cite{GKM06,MP1}.
Since $g\hV_{x,r}=\hV_{gx,r}$, it suffices to recall the definition
for $x\in\Ao$, where it can be constructed in terms of a grading on
$V\otimes_k k[z,z^{-1}]$.  If $\chi\in X^*(T)$ and $V_\chi$ is the
corresponding weight space, then the $r$-th graded subspace is given
by \begin{equation}\label{graded} \hV_{x,\Ao} (r) = \bigoplus_{\langle\chi,x\rangle + m = r} V_\chi z^m \subset \hV.
\end{equation}
The grading depends on the choice of apartment.  However, since we only use gradings with respect to $\Ao$, we usually write $\hV_x(r)$ for $\hV_{x,\Ao}(r)$.  For any $r \in \R$, define
\begin{equation*}
\begin{aligned}
\hV_{x, r} & = \prod_{s \ge r} \hV_{x,\Ao}(s)\subset \hV; &
\hV_{x, r+} & = \prod_{s > r} \hV_{x,\Ao} (s) \subset \hV.
\end{aligned}
\end{equation*}
The collection of lattices $\{\hV_{x, r}\}$ is the Moy-Prasad filtration on $\hV$ associated to $x$.  The set $\Crit_x(V)$ of \emph{critical numbers} of $V$ at $x$ is the discrete, $\Z$-invariant subset of $\R$ for which $\hV_{x, r} / \hV_{x, r+}\cong \hV_{x,\Ao}(r) \ne \{0\}$.  It is easy to see that the sets of critical numbers associated to the adjoint and coadjoint representations coincide and are symmetric around $0$.  

There is also a corresponding filtration $\{\hG_{x,r}\}_{r\in\R_{\ge
    0}}$ of the parahoric subgroup $\hG_x=\hG_{x,0}$ for $x\in\B$. If
one sets $\hG_{x, r+}=\bigcup_{s>r}\hG_{x,s}$, then
$\hG_{x+}=\hG_{x,0+}$ is the pro-unipotent radical of $\hG_x$.  For $r
> 0$, there is a natural isomorphism $\hG_{x, r} / \hG_{x, r+} \cong
\hfg_{x, r} / \hfg_{x, r+}$~\cite{MP1}. On the other hand, $\hG_x/\hG_{x+}$ is
isomorphic to a reductive, maximal rank subgroup of $G$.  For
$x\in\Ao$, we give an explicit isomorphism.  Let $H_x\subset G$ be
the subgroup generated by $T$ and the root subgroups $U_\a$ such that
$d\a(\tx)\in\Z$.  (If $\C\subset k$, $H_x$ is the connected
centralizer of $\exp (2 \pi i \tx)\in G$.)  There is a homomorphism
$\theta'_x:H_x\to\hG_x$ defined on the generators of $H_x$ via
$T\hookrightarrow T(\fo)$ and $\theta'_x(U_\a(c))=U_\a(c
z^{-\a(\tx)})$ for $c\in k$. The induced map $\theta_x:H_x\to
\hG_x/\hG_{x+}$ is an isomorphism~\cite{BrSamink}.  It is easy to see
that the group $H_x$ acts on $\hV_x(r)$ for any $r$ and $\theta_x$
intertwines the representations $\hV_x(r)$ and
$\hV_{x,r}/\hV_{x,r+}$.

We collect the basic properties of these filtrations in the following proposition.

\begin{prop}\label{ev}   Take $V\in\Rep(G)$, and fix $x \in \Ao$ and
  $r \in \R$.
\begin{enumerate}
\item The space $\hV_x(r)$ is the eigenspace 
 corresponding to
the eigenvalue $r$ in $\hV$ for the differential operator $\tau + \tx$.
\item\label{ev2} An element $v \in \hV$ lies in $\hV_{x, r}$ if and only if 
$(\tau + \tx) (v) - r v \in \hV_{x, r+}$.
\item The set $\hV_x(r)$ constitutes a full set of coset representatives for 
the coset space $\hV_{x, r} / \hV_{x, r+}$.
\item If $X\in\hfg_x(s)$, then $\ad(X)(\hV_x(r))\subset \hV_x(r+s)$.
\end{enumerate}
\end{prop}

If $E$ is a degree $e$ extension of $F$, then these gradings and
filtrations extend naturally to $V(E)$ by setting the valuation of the
uniformizer in $E$ to be $1/e$.  For example, if $E=F((u))$ with
$u^e=z$, then $\ft(E)(m/e)=u^m \ft$ for $m\in\Z$.
Proposition~\ref{ev} remains true if one interprets $\tau$ as
$\frac{1}{e} u \frac{d}{d u}$, its unique extension to $E$.  If
$\Gamma=\Gal(E/F)$, then $V_{x,r}=V(E)_{x,r}^\Gamma$ and similarly for
the gradings.

We will frequently need to compare Moy-Prasad filtrations on $\hG$
with filtrations on nonsplit maximal tori.  For a torus, there is a
unique Moy-Prasad filtration on the maximal bounded subgroup and the
Cartan subalgebra.  If $S$ is a maximal torus with Lie algebra $\fs$,
then we may define graded and filtered pieces by extending scalars to
a finite splitting field $E$ and conjugating the analogous data for
$T(E)$ and $\ft(E)$.  To be more explicit, if $g\in G(E)$ satisfies
$\Ad(g^{-1})\fs(E)=\ft(E)$ and $\Gamma=\Gal(E/F)$, then $\fs(r) =
\left(\Ad(g)(\ft(E)(r))\right)^\Gamma$, $\fs_r =
\left(\Ad(g)(\ft(E)_{r})\right)^\Gamma$, and $S_r = \left( g(T(E))_{r
  })g^{-1} \right)^\Gamma$.  These definitions do not depend on the
choice of $E$ or $g$.  Indeed, if one takes another diagonalizer $g'$
and takes $E$ big enough so that $g,g'\in G(E)$, then $g'=gn$ for
$n\in N(E)$.  Independence now follows, since $N(E)$ fixes the grading
and filtrations on $\ft(E)$ and $T(\fo_E)$.  Observe that this
definition rescales the index on filtrations for nonsplit groups
constructed in \cite{MP2} by a factor of $1/e$, but it is effectively
the same as that appearing in \cite[Section 10]{BroSt} and
\cite[Section 5]{GKM06}.

Moy-Prasad filtrations are well-behaved under duality.  If $W$ is an
$\fo$-module, let $W^\vee$ be its smooth ($k$-linear) dual.  Note that
if $V\in\Rep(G)$, then there is a $\hG$-isomorphism
$\widehat{(V^\vee)}=V^\vee\otimes F\overset{\kappa}{\to} (\hV)^\vee$,
$\kappa(\a)(v)=\Res \a(v)\ddz$; we will abuse notation slightly by
denoting both by $\hV^\vee$.  However, $\hV^\vee_{x,r}$ will always
mean $(\hV^\vee)_{x, r}$.  We recall the following facts.  (See
\cite{BrSamink} for more details.)

\begin{prop}\label{dual} Fix $V\in\Rep(G)$, $x\in\B$, and $r\in\R$.
\begin{enumerate}\item The isomorphism $\kappa$ restricts to give
  $\hG_x$-isomorphisms $\hV^\vee_{x,-r}\cong \hV^\perp_{x,r+}$ and
  $\hV^\vee_{x,-r+}\cong \hV^\perp_{x,r}$.  
\item There is a natural $\hG_x$-invariant perfect pairing
  \begin{equation*}
 \hV^\vee_{x, -r}/\hV^\vee_{x, -r+} \times \hV_{x, r} / \hV_{x, r+} \to k,
 \end{equation*} 
 which induces the isomorphism $(\hV_{x, r}/ \hV_{x, r+})^\vee \cong
 \hV^\vee_{x, -r}/\hV^\vee_{x, -r+}$.
\item There are $\hG_x$-isomorphisms $(\hV_{x,r})^\vee\cong
  \hV^\vee/\hV^\vee_{x,-r+}$ and $(\hV_{x,r+})^\vee\cong
  \hV^\vee/\hV^\vee_{x,-r}$.
\item \label{dual4} Suppose that $V$ is endowed with a nondegenerate
  $G$-invariant symmetric bilinear form $\left(,\right)$.  Then,
  $\left(,\right)_\ddz\overset{\text{def}}{=}\Res(\left(,\right)\ddz$ induces $\hG_x$-isomorphisms
  $\hV^\vee_{x,-r}\cong\hV_{x,-r}$ and
  $\hV^\vee_{x,-r+}\cong\hV_{x,-r+}$; in particular, $(\hV_{x, r}/
  \hV_{x, r+})^\vee \cong \hV_{x, -r}/\hV_{x, -r+}$.
\end{enumerate}
\end{prop}

\subsection{Formal flat $G$-bundles and strata}

A formal principal $G$-bundle $\sG$ is a principal $G$-bundle over the
formal punctured disk $\Dx$.  The $G$-bundle $\sG$ induces a tensor
functor from $\Rep(G)$ to the category of formal vector bundles via
$V\mapsto V_\sG = \sG \times_G V$, and this tensor functor uniquely
determines $\sG$.  Formal principal $G$-bundles are trivializable, so
we may always choose a trivialization $\phi : \hG \to \sG$.  Moreover,
there is a left action of $\hG$ on the set of trivializations of
$\sG$.

A flat structure on a principal $G$-bundle is a formal derivation $\n$
that determines a compatible family of flat connections on $V_\sG$ for
all $V \in \Rep(G)$.  In terms of a fixed trivialization $\phi$ for
$\sG$, $\n$ acts on $V_\sG$ as the operator $d + \nbr_\phi \wedge$,
where $d$ is the ordinary exterior derivative and $\nbr_\phi \in
\Om^1_{F}(\hfg)$ is the \emph{matrix} of $\n$ in the trivialization
$\phi$.  Since $\Om^1_{F}(\hfg)\cong\Om^1_{F}(\hfg^\vee)$ via the
choice of invariant form on $fg$ and
$\Om^1_{F}(\hfg^\vee)\cong\hfg^\vee$ canonically, we can view
$\nbr_\phi$ as a functional on $\hfg$.  The group $\hG$ acts on
$\nbr_\phi$ by gauge transformations, namely
\begin{equation}\label{gauge}
\nbr_{g \phi} = g \cdot \nbr_\phi = \Ad^*(g) (\nbr) - (dg) g^{-1}.
\end{equation}
The right-invariant Maurer-Cartan form $(dg) g^{-1}$ lies in
$\Om^1_{F}(\hfg)$.  Note that if $\iota_\tau$ is the inner derivation
by $\tau$, then we can write $[\n]_\phi=[\nt]_\phi\ddz$, where
$[\nt]_\phi=\iota_\tau[\n]_\phi\in\hfg$.  The flat $G$-bundle
$(\sG,\n)$ is called regular singular if the flat connection $V_\sG$
is regular singular for each $V\in\Rep(G)$; otherwise, it is irregular
singular.

We now recall some results from the theory of minimal $K$-types (or
fundamental strata) for formal flat $G$-bundles developed
in~\cite{BrSamink}.  Given $x\in\B$ and a nonnegative real number $r$,
a $G$-stratum of depth $r$ is a triple $(x,r,\b)$ with $\b \in
(\hfg_{x, r}/\hfg_{x, r+})^\vee$.  We say that $\tb \in \hfg^\vee_{x,
  -r}$ is a representative for $\b$ if the coset $\tb+\hfg^\vee_{x,
  -r+}$ corresponds to $\b$ under the isomorphism $\hfg^\vee_{x, -r} /
\hfg^\vee_{x, -r+}\cong (\hfg_{x, r}/\hfg_{x, r+})^\vee$.  If
$x\in\Ao$, we let $\tbo$ denote the unique homogeneous representative
in $\hfg^\vee_x(-r)$.  The loop group $\hG$ acts on the set of strata
with $g\cdot(x,r,\b)$ the stratum determined by $gx$, $r$, and the
coset $\Ad^*(g)(\tb)+\hfg^\vee_{gx, -r+}$.

A stratum is called \emph{fundamental} if $\b$ is a semistable point
of the $\hG_x/\hG_{x+}$-representation $(\hfg_{x,
  r}/\hfg_{x,r+})^\vee$; equivalently, the corresponding coset
$\tb+\hfg^\vee_{x,-r+}$ does not contain a nilpotent element.  This
can only occur when $r\in\Crit_x(\fg)$.  If $x\in\Ao$, then a stratum
is nonfundamental if and only if the homogeneous representative $\tbo$
is nilpotent.

Given $x \in \Ao$, we say that the flat $G$-bundle $(\sG, \n)$
contains the stratum $(x,r, \b)$ with respect to the trivialization
$\phi$ for $\sG$ if $[\n]_\phi - \tx \ddz \in \hfg_{x,r+}^\perp$ and
the coset $\left([\n]_\phi - \tx \ddz \right) + \hfg_{x, -r+}^\vee$
determines the functional $\b \in (\hfg_{x, r} / \hfg_{x,r+})^\vee$.
Note that if $r>0$, then $\b$ is the functional determined by
$[\n]_\phi$; moreover, if $x'$ and $x$ have the same image in $\bB$,
then $(x',r,\b)$ is also contained in $(\sG, \n)$.

Given a flat $G$-bundle $(\sG, \n)$, we say that its \emph{slope} is
the infimum of the depths of the strata contained in it.  In
\cite{BrSamink}, it is shown that this infimum is actually attained
and that it is a rational number.  More precisely, we have the
following theorem.

\begin{thm}\cite[Theorem 3.15]{BrSamink}\label{MP}
  Every flat $G$-bundle $(\sG, \n)$ contains a fundamental stratum
  $(x,\slope(\sG),\b)$, where $x$ is an optimal point in $\Ao$ in the
  sense of \cite{MP1}; the slope is positive if and only if $(\sG,
  \n)$ is irregular singular.  Moreover, the following statements
  hold.
\begin{enumerate}
\item\label{MP1} If $(\sG, \n)$ contains the stratum $(y,r', \b')$, then
$r' \ge \slope(\sG)$.  
\item \label{MP3} If $(\sG,\n)$ is irregular, a stratum $(y,r', \b')$
  contained in $(\sG, \n)$ is fundamental if and only if $r' = \slope(\sG)$.
\end{enumerate}
\end{thm}

In particular, the slope is an \emph{optimal number}--a critical
number for an optimal point in $\Ao$.

For future reference, we recall the following lemma from
\cite{BrSamink} describing the calculus for change of trivialization on
strata contained in $\sG$.

\begin{lem}\cite[Lemma 3.4]{BrSamink}\mbox{}\label{actlem}
\begin{enumerate}
 \item \label{act1}  If $n \in \hN$, $[\n]_{n \phi}
   - \widetilde{n x} \ddz \in \Ad^*(n) ( [\n]_{ \phi } - \tx \ddz) +
   \hft_{0+}\ddz$.
\item\label{act2} If $X \in \hfu_\a \cap \hfg_{x, \ell}$,
then
\begin{equation*}
[\n]_{\exp(X) \phi} - \tx \ddz \in \Ad^*(\exp(X)) ([\n]_\phi - \tx \ddz) -
\ell X \ddz +
\hfg^\vee_{x, \ell+}.
\end{equation*}
\item\label{act3} If $p \in \hG_{x}$, then
$[\n]_{p \phi} -  \tx \ddz \in \Ad^*(p)([\n]_\phi - \tx \ddz)  +
\hfg^\vee_{x+}$.
\item\label{act4} If $p \in \hG_{x, \ell}$ for $\ell > 0$, then
$[\n]_{p \phi} -  \tx \ddz \in \Ad^*(p)([\n]_\phi - \tx \ddz)  +
\hfg^\vee_{x, \ell}$.

\end{enumerate}
\end{lem}

\begin{rmk} Applying part~\eqref{act3} of the lemma to $[\n]_\phi=0$,
  we see that if $p\in\hG_x$, then
  $\tau(p)p^{-1}\in\Ad(p)(\tx)-\tx+\hfg_{x+}\subset\hfg_x$. This fact
  will be used throughout the paper.
\end{rmk}

\section{Compatible Filtrations}\label{torus} 

Intuitively, one can view a fundamental stratum contained in a flat
$G$-bundle as a nondegenerate ``leading term'' of the derivation $\n$.
The goal of this paper is to study flat $G$-bundles containing strata
corresponding to regular semisimple leading terms.  In order to do
this, we need to study filtrations that are compatible with the
natural filtration on a maximal torus in $\hG$.

By \cite[Lemma 2]{KL}, there is a bijection between
the set of conjugacy classes of Cartan subalgebras in $\hfg$ (resp.
maximal tori in $\hG$) and the set of conjugacy classes in $W$.  We
briefly recall the correspondence.  Let $\bG\cong\hat{Z}$ be the
absolute Galois group of $F$.  If $\fs\subset \hfg$ is a Cartan
subalgebra, then there exists $g\in G(\bF)$ such that
$\Ad(g)\ft(\bF)=\fs(\bF)$, so $\rho\mapsto g^{-1}\rho(g)$ is a
$1$-cocycle of $\bG$ with values $N(\bF)$.  In fact, since $H^1(F,\hG)=1$ (as
$G$ is connected reductive and $F$ has cohomological dimension $1$),
all $1$-cocycles in $N(\bF)$ are of this form, and such a cocycle
coming from $h\in G(\bF)$ gives rise to the Cartan subalgebra
$(\Ad(h)(\ft(\bF)))^{\bG}$.  The induced map gives a bijection between
$H^1(F,\hN)$ and the conjugacy classes of Cartan subalgebras.
Moreover, $H^1(F,\hN)$ is isomorphic (as pointed sets) to the set of
conjugacy classes of $W$; the image of the above cocycle is the class
of $g^{-1}\sigma(g)T(\bF)\in W$, where $\sigma$ is a fixed topological
generator of $\bG$.  In particular, $\Ad(g^{-1})$ intertwines the
action of $\sigma$ on $\fs(\bF)$ with $w\circ\sigma$ on $\ft(\bF)$.
Of course, if $E$ is a finite splitting field for $\fs$, then one can
take $g\in G(E)$ and $\Ad(g^{-1})$ again intertwines a fixed generator
$\sigma$ for $\Gal(E/F)\cong \Z_{[E:F]}$ with $w\circ\sigma$.  Since
$w$ and $\sigma$ commute, it follows that the order of $w$ divides
$[E:F]$.

\begin{defn}\label{cc}
  Let $\g$ denote a conjugacy class in $W$.  We say that a Cartan
  subalgebra $\fs \subset \hfg$ (resp. a maximal torus $S \subset
  \hG$) is of type $\g$ if the conjugacy class of $\fs$ (resp. $S$)
  corresponds to $\g$ as above.
\end{defn}

For the remainder of this section, we fix a maximal torus $S$ type
$\gamma$, and choose a representative $w\in W$ for $\gamma$.  Let
$E=k((z^{1/e}))$ be a splitting field of $\fs$, and let
$\sigma\in\Gamma=\Gal(E/F)$ be a fixed generator.  We then take $g\in
G(E)$ such that $\Ad(g^{-1})\fs(E)=\ft(E)$ and
$g^{-1}\sigma(g)T(E)=w$.  We call such a $g$ a \emph{$w$-diagonalizer}
of $S$.

\subsection{Compatible gradings and filtrations}

Let $S$ be a maximal torus in $\hG$ with Cartan subalgebra $\fs$.

\begin{defn}
  We say that a point $x \in \BT$ is \emph{compatible} with $\fs$ if
  the filtration induced by $x$ on $\fs$ is the (rescaled) Moy-Prasad
  filtration on $\fs$, i.e., $\fs_r=\hfg_{x,r}\cap\fs$ for all $r$.
  If $x\in\Ao$, $x$ is \emph{graded compatible} with $\fs$ if
  $\fs(r)=\hfg_x(r)$ for all $r$ and $\fs(0) \subset \hft(0)=\ft$.
\end{defn}

\begin{rmk}  Similarly, we say that $x\in\B$ is compatible with the
  maximal torus $S$ if $S_r=\hG_{x,r}\cap S$ for all $r\ge 0$.  It is
  easy to see that $x$ is compatible with $S$ if and only if it is
  compatible with $\fs$.
\end{rmk}

It is obvious that every point in $\Ao$ is graded compatible (hence
compatible) with $\hft$.  Conversely, if $\fs'$ is a split Cartan
subalgebra graded compatible with $x\in\Ao$, then
$\dim\fs'(0)=\dim\ft$ and $\fs'(0)\subset\ft$, so $\fs'=\hft$.  As we
will see later, nonsplit Cartan subalgebras can be graded compatible
with points in $\Ao$.  One can also show that if $\fs$ is any Cartan
subalgebra compatible with each point of $\Ao$, then $\fs=\hft$.
Indeed, it follows from the definitions that
$\fs_0\subset\cap_{x\in\Ao}\hfg_{x,r}=\hft_0$, and since $\fs_0$
contains a regular semisimple element, $\fs=\hft$.

The goal of this section is to show that if $x\in\Ao$ is compatible
with $S$, then there exists a $\hG_x$-conjugate of $S$ which is graded
compatible with $x$.  We begin with several equivalent formulations of
graded compatibility.

\begin{lem} \label{lem1} 
  The following statements are equivalent:\begin{enumerate}
\item  The point $x\in\Ao$ is graded compatible with $\fs$;
\item  $\left(\tau + \ad(\tx)\right) (\fs) \subset \fs$;
\item\label{lem13} If $g\in G(E)$ is a $w$-diagonalizer, then
  $\Ad(g^{-1}) (\tx + \tau (g) g^{-1}) \in \ft(E)$; and
\item If $g\in G(E)$ is a $w$-diagonalizer, then $\Ad(g^{-1}) (\tx +
  \tau (g) g^{-1}) \in \ft(\fo_E)$.
\end{enumerate}
\end{lem}

\begin{proof}

Suppose $x$ is graded compatible with $\fs$.  Then, $\fs(r)
= \fs \cap \hfg_x (r)$ implies that there is a topological basis for
$\fs$ consisting of eigenvectors for $\tau + \ad(\tx)$, proving the
second statement.

Next, observe that
\begin{equation}\label{gaugetau}
\left(\tau+ \ad(\tx)\right) (\Ad(g) X) = \Ad(g)\left( \left[\tau +
    \ad\left(\Ad(g^{-1})(\tx) +  g^{-1}\tau (g)\right)\right] (
  X)\right)
\end{equation} for $X\in\fg(E)$.
If $\left(\tau + \ad(\tx)\right) (\fs) \subset \fs$, then applying
$\Ad(g^{-1})$ of
this equation to $X\in\ft(E)$ gives
\begin{equation}\label{lemeqn1}
  \left[ \tau + \ad \left( \Ad(g^{-1})(\tx) + g^{-1} \tau (g)  \right) \right] (\ft(E)) \subset \ft(E).
\end{equation}
Since
$\tau \ft(E) \subset \ft(E)$, it follows that $\ad \left(
  \Ad(g^{-1})(\tx) +g^{-1}  \tau (g) \right) (\ft(E)) \subset \ft(E)$.
Therefore, $\Ad(g^{-1}) (\tx) + g^{-1}\tau (g)  \in \ft(E)$. 

Now, assume that $\Ad(g^{-1}) (\tx) + g^{-1}\tau (g) \in \ft(E)$.  We
see that the differential operator $\tau+ \ad \left( \Ad(g^{-1})(\tx)
  + g^{-1} \tau (g) \right)$ restricts to $\tau$ on $\ft(E)$, so the
$r$-eigenspace of the former on $\ft(E)$ is $\ft(E)(r)$.  Applying
\eqref{gaugetau} and Proposition~\ref{ev} gives
$\Ad(\ft(E)(r))=\fs(E)\cap \fg(E)_x(r)$, and $\fs(r)=\fs \cap
\hfg_x(r)$ follows by taking Galois fixed points.

It remains to show the equivalence of the last two statements.  One
direction is trivial, so assume that $Y = \Ad(f^{-1}) (\tx + \tau (f)
f^{-1})\in\ft(E)$.  By the Iwasawa decomposition, we can write $f^{-1}
= put$, where $p\in G(\fo_E)$, $u\in U(E)$ (with $U$ the unipotent
radical of $B$), and $t\in T(E)$.  Before calculating $Y$, we make
several observations.  First, the fact that $\tau(z^m)z^{-m}=m$
implies that
$\tau(t)t^{-1}=t^{-1}\tau(t)\in\ft(\Q)+z^{1/e}\ft(\fo_E)$.  Setting
$X=\Ad((ut)) [ \tx + \tau ((ut)^{-1}) (ut)]$, we see that
$X=\Ad(u)(\tx-\tau(t)t^{-1}-u^{-1}\tau(u))\in\tx-\tau(t)t^{-1}+\fu(E)\subset
\fg(\fo_E) + \fu(E)$.

Next, write $p^{-1} = p_1 p_2$ with $p_1 \in G$ and $p_2$ in the first
congruence subgroup of $G(\fo_E)$ with respect to $z^{1/e}$ (i.e.,
$G(E)_{o+}$).  We obtain $\tau (p^{-1}) p=\tau (p_1 )
p_1^{-1}+\Ad(p_1)(\tau (p_2) p_2^{-1})= \tau (\log (p_2)) \in
z^{1/e}\fg(\fo_E)$.  Since $Y=\Ad(p)(X+\tau (p^{-1}) p)$, we see that
$\Ad(p)(X+\fg(\fo_E))=Y + \fg(\fo_E)$.  Because $X+\fg(\fo_E)$ contains
the $\ad$-nilpotent element $-\tau(u)u^{-1}$, $Y + \fg(\fo_E)$ also
contains an $\ad$-nilpotent element.

Suppose that $Y\notin \fz(E) + \ft(\fo_E)$.  Let $n>0$ be the smallest
integer such that $Y \in \fz(E) + z^{-n} \ft(\fo_E)$.  This means that
there exists a root $\alpha$ such that $\a(Y) \in z^{-n} \fo_E
\setminus z^{1-n} \fo_E$.  Thus, the action of $\ad(Y)$ on the root
subalgebra $\fu(E)_\a$ is non-nilpotent.  Furthermore, if $Y' \in Y +
\fg(\fo_E)$ and $Z \in \fu_\a$, then $(\ad(Y'))^n (Z) \in (\ad(Y))^n
(Z) + \fg (\fo_E) =\a(Y)^n Z + \fg(\fo_E)$ and hence is nonzero.
Thus, no element of $Y + \fg(\fo_E)$ is $\ad$-nilpotent, a contradiction.

Accordingly, $X\in\Ad(p^{-1})(Y+\fg(\fo_E))\subset \fz(E)\cap
\fg(\fo_E)$.  This means that $X\in (\fz(E)\cap
\fg(\fo_E))\cap(\fu(E)+\fg(\fo_E))=\fg(\fo_E)$, so $Y\in\fg(\fo_E)$ as
well.
\end{proof}

\begin{lem}\label{c2}
  A Cartan subalgebra $\fs$ is compatible with $x \in \Ao$ if and only
  if $\fs(E)_r=\fg(E)_{x,r}\cap\fs(E)$ for all $r\in\R$.
\end{lem}

\begin{proof}

The reverse implication follows by taking Galois
invariance of the equations $\fs(E)_r=\fs(E)\cap\fg(E)_{x,r}$.

Now, suppose that $S$ is compatible with $x$.  Let $\tr:\fs(E)\to\fs$
be the trace map, so that $\eta_i(X)=\frac{1}{e}z^{i/e}\tr(z^{-i/e}X)$
is the projection onto the $\xi^i$-eigenspace for $\s$.  Since
$z^{i/e}\fs(E)_r=\fs(E)_{r+\frac{1}{e}}$ and
$z^{i/e}\fg(E)_{x,r}=\fg(E)_{x,r+\frac{1}{e}}$, we obtain
$\eta_i(\fs(E)_r)\subset z^{i/e}\fs_{r-\frac{1}{e}}\subset
z^{i/e}\fg_{x,r-\frac{1}{e}}\subset\fg(E)_{x,r}$.  The
action of $\Gamma$ is completely reducible, so $\fs(E)_r\subset\fg(E)_{x,r}$.
On the other hand, suppose that there exists $X\in(\fs(E)\cap\fg(E)_{x,r})\setminus
\fs(E)_r$.  The same must be true for $\eta_i(X)$ for some $i$.  We
then obtain $z^{-i/e}\eta_i(X)\in(\fs\cap\fg_{x,r-\frac{1}{e}})\setminus
\fs_{r-\frac{1}{e}}$, a contradiction.
\end{proof}

If $S'\subset\hG$ is a maximal split torus compatible with $x\in\Ao$,
an elementary version of the argument given below in
Proposition~\ref{r=0} shows that there exists $g\in\hG_x$ such that
$g^{-1}S' g=\hT$.  If $x$ is compatible with $S$, it follows from this
and the previous lemma that there exists $p\in G(E)_x$ satisfying
$p^{-1}S(E) p=T(E)$.  We will need a refinement of this statement.

\begin{lem}\label{c3}
  Suppose that $S \subset \hG$ is a maximal torus that splits over
  $E$.  If $x \in \Ao$ is compatible with $\fs$, then there exists $q
  \in G(E)_{x}$ such that $q^{-1} S(E) q = T(E)$ and $q^{-1} \s(q) \in
  N$.
\end{lem}

\begin{proof}

  With $p\in G(E)_x$ as defined in the preceding paragraph, we will
  construct $s\in T(E)_{0+}$ such that $q=ps$ satisfies $q^{-1} \s(q)
  \in N$. Noting that $p^{-1} \s (p) \in \hN(E) \cap G(E)_{x}$, we can
  write $p^{-1} \s (p) = t n$ with $n \in N$ and $t \in T(E)_{0+}$.
  An easy induction gives $\s^j(p)=p\prod_{i = 0}^{j-1} \s^i (tn)$ for
  $j\ge 0$ (the product taken in increasing order); since $\s^e(p)=p$,
  we obtain $1=\prod_{i = 0}^{e-1} \s^i (tn)$.  It follows that
  $\prod_{i = 0}^{e-1}(w \circ \s)^e (t)\in T \cap G(E)_{x+} = \{1\}$.

  Set $M = T(E)_{0+}$, and view it as a $\Z/e\Z$-module with $\bar{1}$
  acting as $w\circ \s$.  If $h$ denotes the norm map for this action,
  the previous paragraph shows that $h(t)=1$.  We will show that
  $H^1(\Z/e\Z,M)=\{1\}$.  Assuming this, there exists $s\in M$ such
  that $s^{-1} (w \circ \s) (s) = t^{-1}$.  We then see that
  $(ps)^{-1}\s(ps) = s^{-1} w \circ \s(s) p^{-1} \s (p) = n$.

  The exponential map gives an equivariant isomorphism $\ft(E)_{0+}\to
  M$.  Thus, to show that $H^1(\Z/e\Z,M)=\{1\}$, it suffices to check
  that $H^1(\Z/e\Z,\ft(E)_{0+})=\{0\}$.  Recalling that a
  $w$-diagonalizer intertwines this action with the usual Galois
  action on $\fs(E)_{0+}$, we are reduced to showing that
  $H^1(\Gamma,\fs(E)_{0+})=\{0\}$.  If $X$ is in the kernel of the
  norm map, there exists an element $Y'\in\fs(E)$ such that
  $\s(Y')-Y'=X$ by Hilbert's Theorem 90. Letting $Y\in\fs(E)_{0+}$ be
  the projection of $Y$ obtained by killing graded terms in
  nonnegative degree, one has $\s(Y)-Y=X$ as desired.
\end{proof}

We define $\pi_\fs$ to be the orthogonal projection of $\hfg$ onto
$\fs$ with respect to the $F$-bilinear invariant pairing $\left<,
\right>$ on $\hfg$ obtained by extending scalars.  The invariance of
the form makes it clear that $\pi_{\Ad(g) (\fs)} = \Ad(g) \circ
\pi_\fs \circ \Ad(g^{-1})$ for any $g \in \hG$.  We will analyze this
map in detail in Section~\ref{sec:RS}.

\begin{thm}\label{compthm}
  Suppose that $x\in\Ao$ is compatible with the Cartan subalgebra
  $\fs$.  Then, there exists $q \in \hG_x$ such that $x$ is graded
  compatible with $q^{-1} \fs q$.
\end{thm}
\begin{proof}

  Applying Lemma~\ref{c3}, choose $p \in G(E)_{x}$ such that
  $p^{-1}S(E) p =T(E)$ and $p^{-1} \s (p) = n \in N$.  In particular,
  $p$ is a $w$-diagonalizer of $S$ for $w=nT\in W$.  Furthermore,
  $\tau (p) p^{-1} \in \hfg_x$, since $\s(p) = p n$ and $\s (\tau (p)
  p^{-1}) = (\tau (p) n) (n^{-1} p^{-1}) = \tau(p) p^{-1}$.  Finally,
\begin{equation}\label{comp1}
\tau(p) p^{-1} + \tx \in \Ad(p) (\tx) + \hfg_{x+}.
\end{equation}  This follows from
Lemma~\ref{actlem}\eqref{act3} by setting $[\n]_\phi$ to be the zero form
and applying the inner derivation $\tau$.

If $\tau(p) p^{-1} + \tx \in \fs$, then $x$ is graded compatible with
$\fs$ by Lemma~\ref{lem1}\eqref{lem13}.
Therefore, it is enough to show that there exists $q \in \hG_{x+}$
(so that $qp\in G(E)_{x}$ will be
a $w$-diagonalizer for $qSq^{-1}$ with $(qp)^{-1}\s(qp)=n$) such that
$\tau (qp) (qp)^{-1} + \tx \in \Ad(q) (\fs(E))$.  Equivalently, if $q' =
p^{-1} q p$,
\begin{equation*}
q'^{-1}\tau (q')  + \Ad(q'^{-1}) ( \Ad(p^{-1}) \left[\tau(p)p^{-1} + \tx
\right]) \in \ft(E).
\end{equation*}
If we set $X=\Ad(p^{-1}) \left[\tau(p)p^{-1} + \tx \right]-\tx$, we
see from \eqref{comp1} that $X\in\Ad(p)(\hfg_{x+})\subset\fg(E)_{x+}$.
Also, since $\tau(p)p^{-1}\in\hfg_x$, $X+\tx\in\Ad(p^{-1})(\hfg_x)$.
We have thus reduced to the following general problem: Given $X \in
\fg(E)_{x+}$ such that $X+\tx \in \Ad(p^{-1}) (\hfg_{x})$, find $q'
\in p^{-1}\hG_{x+}p$ such that $\Ad(q'^{-1}) (X + \tx) +q'^{-1} \tau
(q') \in \ft(E)$.

We will construct $q'$ recursively.  First, observe that since
$\Ad(p^{-1})$ intertwines the Galois action with the twisted Galois
action generated by $\s' = \Ad(n)\circ\s$, $X+\tx$ is fixed by $\s'$.
This action preserves the $x$-filtration, so $\tx$ and $X$ are
individually fixed.  Thus, $X+\tx\in\Ad(p^{-1})(\fs_{0}) +
\Ad(p^{-1})(\hfg_{x, \ell_0})$ with $\ell_0>0$.

We now apply an inductive argument.  Suppose that $Y\in \fg(E)_{x+}$
satisfies $Y + \tx \in \Ad(p^{-1})(\fs_{0}) + \Ad(p^{-1})(\hfg_{x,
  \ell})$ for $\ell > 0$ a critical number.  Let $Y+ \tx
=\pi_{\hft}(Y+\tx) + \sum_{\psi \in \Phi} Y_{\psi,\ell}$ where
$Y_{\psi, \ell} \in \fu_\psi \cap \fg(E)_{x, \ell}$.  Since
$\Ad(p^{-1})\pi_{\fs}=\pi_{\hft}\Ad(p^{-1})$, it follows that
$\pi_{\hft}(Y+\tx)\in\Ad(p^{-1})(\fs_{0})$ and hence $\sum_{\psi \in
  \Phi} Y_{\psi,\ell}\in \Ad(p^{-1}) (\hfg_{x, \ell})$.

Write $q_{\psi,\ell} = \exp (-\frac{1}{\ell} Y_{\psi,\ell}) $.
Proposition~\ref{ev} now gives
\begin{align*}
\Ad(q_{\psi, \ell}^{-1}) (Y + \tx) + q_{\psi, \ell}^{-1} \tau
(q_{\psi,\ell}) & \in
-\left[\ad(\tx) (\tfrac{1}{\ell} Y_{\psi,\ell}) + \tau (\tfrac{1}{\ell}
Y_{\psi,\ell}) \right] +
Y + \tx + \fg(E)_{x, \ell+} \\
& = Y + \tx - Y_{\psi,\ell} + \fg(E)_{x, \ell+}.
\end{align*}
Set $q_\ell=\exp (-\sum_{\psi \in \Phi} \frac{1}{\ell} Y_{\psi, \ell}
)\in\exp(\Ad(p^{-1})(\hfg_{x,\ell}))=p^{-1}\hG_{x,\ell}p$.  It is
evident from the calculation above that $\Ad(q_\ell^{-1})(Y + \tx) +
q_\ell^{-1} \tau(q_\ell) \in \Ad(p^{-1})(\fs_{0}) +\fg(E)_{x, \ell'}$
for some critical number $\ell' > \ell$.

In order to show that $Y' = \Ad(q_\ell^{-1})(Y + \tx) + q_\ell^{-1}
\tau(q_\ell) - \tx$ satisfies the inductive hypothesis, first observe
that $\Ad(q_\ell^{-1})(Y)$, $q_\ell^{-1} \tau(q_\ell)$, and
$\Ad(q_\ell^{-1})(\tx)-\tx$ all lie in $\hfg(E)_{x+}$, so the same
holds for $Y'$.  Since we already know that $Y'+\tx\in
\Ad(p^{-1})(\fs_{0}) +\Ad(p^{-1})(\fg(E)_{x, \ell'})$, to show that
$Y'+\tx\in\Ad(p^{-1})(\hfg_{x,\ell'})$, it suffices to show that
$\Ad(q_\ell^{-1})(Y + \tx)$ and $q_\ell^{-1} \tau(q_\ell) $ both lie
in $\Ad(p^{-1}) (\hfg_{x})$.  It is clear that this holds for the
first expression, since $q_\ell^{-1} \in p^{-1}\hG_{x, \ell} p$.  A
direct calculation, using the same fact along with the observation
that $\tau (p) p^{-1} = -p \tau(p^{-1}) \in \hfg_x$, proves the
statement for $ q_\ell^{-1} \tau(q_\ell) $.

We thus obtain an increasing sequence $\ell_i>0$ of critical numbers
for the $x$-filtration and $q_{l_i}\in \hG_{x,\ell_i}$ for which
$q'=\lim(q_{\ell_m}\cdots q_{\ell_1})$ satisfy the desired
condition. 
\end{proof}

\subsection{Compatible points}
 
The goal of this section is to describe the collection of points in
$\Ao$ that are compatible with some Cartan subalgebra of type $\g$.
We begin by showing that if $\fs$ is graded compatible with $x$, then
the $w$-diagonalizer $g$ can be chosen to have very special
properties.  In particular, this construction gives a well-behaved
embedding of a Cartan subalgebra of type $\g$ in $\hfg$.  We then use
this to classify $x\in\cA_0$ which are graded compatible with $\fs$.
Finally, we find all points in $\Ao$ which are compatible with some
conjugate of $\fs$. We remark that in \cite{GKM06}, Goresky, Kottwitz,
and MacPherson construct a particular Cartan subalgebra $\fs$ of type
$\gamma$ and a point $x$ in $\cA_0$ which is graded compatible with
$\fs$.

\begin{lem}\label{fs}
  Suppose $x\in\Ao$ is graded compatible with $\fs$.  Then $y\in\Ao$
  is graded compatible with $\fs$ if and only if $\tx - \ty \in
  \fs(0)$.
\end{lem}
\begin{proof}

  Assume that $x,y\in\Ao$ are graded compatible with $\fs$.  If $s \in
  \fs(r)$, then $\ad(\tx) (s) = \ad(\ty) (s)$, i.e., $\ad(\tx-\ty) (s)
  = 0$.  We may construct a topological basis for $\fs$ consisting of
  elements of $\fs(r)$ for $r\in\frac{1}{e}\Z$.  It follows that
  $\ad(\tx - \ty) (\fs) = \{0\}$, so $\tx - \ty \in \fs$.  Write $\tx
  - \ty=\sum_{r\gg-\infty}s_r$ with $s_r\in\fs(r)$.  We obtain
  $\sum_{r\gg-\infty}r s_r=(\tau+\ad(\tx))(\tx -
  \ty)=\ad(\tx)(\tx-\ty)=[\tx,\ty]$.  Similarly, $\sum_{r\gg-\infty}r
  s_r=(\tau+\ad(\ty))(\tx - \ty)=\ad(\ty)(\tx-\ty)=[\ty,\tx]$.  This
  implies that $s_r=0$ unless $r=0$, so $\tx-\ty\in\fs(0)$.

The converse is immediate from Lemma~\ref{lem1}.
\end{proof}

\begin{lem}\label{vandermonde} 
  The point $x\in\Ao$ is graded compatible with $\fs$ if and
  only if $\fs(0) \subset \ft$ and
  there exists a splitting field $E$ for $\fs$ and a $w$-diagonalizer $g \in G(E)$ of $S$ such that
  $-\tau (g) g^{-1} \in (\tx + \fs(0)) \cap \ft_\Q$.
\end{lem}

\begin{proof}
  Let $f\in G(E')$ be a $w$-diagonalizer of $S$.  If $E$ is a finite
  extension of $E'$, it is obvious that any element of the left coset
  $fT(E)$ is also a $w$-diagonalizer.  We will construct a finite
  extension $E$ of $E'$ and $g\in fT(E)$ that satisfies the property
  in the statement.

We retain the notation of the proof of Lemma~\ref{lem1} with $f$
replacing $g$, so $f^{-1}=put$ and $Y = \Ad(f^{-1}) (\tx + \tau (f) f^{-1}) \in
\ft(\fo_{E'})$.  Write $Y= t_0 + t'$, with $t_0 \in \ft(E')(0)=\ft$ and $t' \in
\ft(E')_{0+}$.  Replacing $f$ with $fe^{v'}$, where $v'\in\ft(E')_{0+}$
satisfies $\tau (v') = - t'$, we may assume that $Y = t_0 \in \ft$.

We now construct $v\in\ft_\Q$ such that $g=f z^v$ satisfies $\tx +
\tau (g) g^{-1} \in \fs(0)$.  If $e v\in\ft_\Z$ and $E$ is an
extension of $E'$ with $[E:F]=e$, then $z^v\in T(E)$.  Write $\tw =
f^{-1} \s (f) \in N(E')$.  Note that we can decompose $\tw = t_1 \tw_2
$, where $\tw_2 \in N$ and $t_1 \in T(E')$.  It follows that $\tau
(\tw) \tw^{-1} = \tau (t_1) t_1^{-1} \in \ft_\Q + \ft(E')_{0+}$.
Moreover, using the facts that $\s$ commutes with $\tau$ and fixes
$\ft$,
\begin{align*}
  \tau (\tw) \tw^{-1} & = \Ad(f^{-1}) \left( -\tau (f) f^{-1}) + \s\left( \tau (f) f^{-1} \right) \right) \\
  & = \Ad(f^{-1}) \left( -\left[\tx + \tau (f) f^{-1} \right] +
    \s\left[
      \tx + \tau (f) f^{-1} \right] \right) \\
  &= -Y+w\circ\sigma(Y)=-Y+w\cdot Y\in\ft.
\end{align*}
We deduce that $\tau (\tw) \tw^{-1} \in \ft_\Q$.  Choose a $\Q$-linear
projection $\ft\to\ft_\Q$, and let $v\in\ft_\Q$ be the image of $Y$.
Then, $\tau (\tw) \tw^{-1} = v - w\cdot v$.

Set $g=f z^v$, so that $\tau (g) g^{-1} = \tau(f) f^{-1} + \Ad(f)
(v)$.  We obtain
\begin{align*}
  \s\left[\tau (g) g^{-1} \right] &= \s(\tau(f) f^{-1})+\s(\Ad(f)
  (v))\\ &=(\tau(f)f^{-1}+\Ad(f)(\tau(\tw)\tw^{-1})+\Ad(f)(w\cdot
  v)\\&= \tau (f) f^{-1} + \Ad(f) ((t - w\cdot v) + w\cdot v) = \tau
  (g) g^{-1}.
\end{align*}
It follows that $\tx + \tau (g) g^{-1}\in(\fs(E)(0))^\Gamma=\fs(0)$.

It remains to show that $\tau (g) g^{-1}\in\ft_\Q$.  First, note that
$Y$ is conjugate to $\tx+\tau (g) g^{-1}\subset \ft+\fs(0)\subset\ft$
by hypothesis.  Next, recall from the proof of Lemma~\ref{lem1} that
$Y\in\Ad((p_1p_2)^{-1})(\tx+q+n+z^{1/e}\fg(\fo_E))$ for some
$q\in\ft_\Q$ and $n\in\fu$, where $p_1\in G$ and $p_2\in G(E)_{o,0+}$.
Since the projection $\fg(\fo_E)\to\fg$ restricts to the identity on
$\fg$ and $Y\in\ft$, applying the projection shows that $Y$ is
conjugate in $G$ to $\tx+q+n\in\fb$.  The latter element is
semisimple, so it is $B$-conjugate to an element of $\ft$, which must
clearly be $\tx+q$.  It follows that $\tx+\tau (g) g^{-1}$ and $\tx+q$
are $G(\fo_E)$-conjugate, so the same is true for $q$ and $\tau (g)
g^{-1}$.  However, two elements of $\ft$ are conjugate only if they
are in the same $W$-orbit, and $W\cdot\ft_\Q=\ft_\Q$.  Thus, $\tau (g)
g^{-1}\in\ft_\Q$.

We now prove the converse.  Since $\tx + \tau (g) g^{-1}\in\fs$,
$\Ad(g^{-1})(\tx + \tau (g) g^{-1})\in\ft(E)$, and the result follows
from Lemma~\ref{lem1}.
\end{proof}

\begin{lem}\label{bnorm} 
Suppose that $n \in \hN$ determines the element $w \in \hW$.
 For all $x \in \Ao$, $u = \Ad(n) (\tx) - \tau (n) n^{-1} \in \widetilde{w x}
+ \hft_{0+}.$ Moreover, if $u \in
 \ft_\R$ and $V$ is a finite-dimensional representation of $G$, then 
 $n\hV_x(r) =\hV_{w x} (r)$.
 \end{lem}
\begin{proof}
  The fact that $u\in \widetilde{w x} + \hft_{0+}$ is proved in
  \cite[Lemma 2.3]{BrSamink}.  Now, assume that $u\in\ft_\R$ (so
  $u=\widetilde{wx}$), and take $X\in V_x(r)$.  Applying
  Lemma~\ref{ev} gives
\begin{align*}
\tau (n X) + u (n X) &=  
\tau (n) n^{-1} (nX) + n (\tau (X)) + u (nX) 
\\
& =n\tau(X) +(\Ad(n)\tx) (nX)=n ((\tau +\tx) (X))  \\
& =r (n X).
\end{align*}
Therefore, $n \hV_{x} (r) \subset \hV_{wx}(r)$.  A similar argument
shows the reverse inclusion.
\end{proof}

We now show that every conjugacy class of Cartan subalgebra in $\hfg$
has a representative $\fs$ that is graded compatible with a point
$x\in\Ao$.  Given $w\in W$, set $\ft^w:=\{X\in\ft\mid wX=X\}$.

\begin{thm}\label{torcon} There exists a maximal torus $S$ of type $\g$
  and a representative $w\in W$ such that $\fs$ is graded compatible with some
  $x\in\Ao$ and satisfies $\fs(0)=\ft^w$.

\end{thm}
\begin{proof} Let $n' \in N$ be a finite order coset representative
for $w'\in\g$; this exists by a theorem of Tits~\cite{Tits66}.    The
element $n'$ is semisimple and commutes with $T^{w'}$.  It follows that
$n'$ and $T^{w'}$ are contained in a maximal torus $T'$.  Choose $h\in
G$ such that $T=hT'h^{-1}$.  Since $n'$ has finite order, there exists
$t \in \ft(\Q)$ such that $\exp (-2 \pi i t) = \Ad(h) (n')$.  Write $g =
z^{-t} h$.  It is clear that $g^{-1} \s(g) = n'$, so $g$ is a
$w'$-diagonalizer of a maximal torus $S \subset \hG$ with type $\g$.  

Let $x$ be the image of $t$ in $\Ao$; equivalently,
$\tx=\widetilde{\tau(g)g^{-1}}$.  Therefore,
$\Ad(g^{-1})(\tx+\tau(g)g^{-1})\in\Ad(g^{-1})(\fz)\subset \ft$.  By
Lemma~\ref{lem1}, $x$ is graded compatible with $\fs$.

It will be shown in Theorem~\ref{gx} below that
$\fs(0)=\Ad(h)(\ft^{w'})$.  A standard argument now shows that there
exists $n_0\in N$ such that $\fs(0)=\Ad(n_0)(\ft^{w'})$.  Indeed,
$\ft$ and $\Ad(h)(\ft)$ are two Cartan subalgebras in $Z(\fs(0))$, so
there exists $k\in Z(\fs(0))$ such that $\Ad(kh)(\ft)=\ft$.  One may
then take $n_0=kh$.  Setting $w=w_0 w' w_0^{-1}\in\g$ with $w_0=n_0
T\in W$, we have $\fs(0)=\ft^w$ as desired.
\end{proof}

\begin{rmk}
  The existence of a Cartan subalgebra $\fs$ of type $\g$ which is
  graded compatible with some $x\in\Ao$ may also be derived from the
  construction in \cite[Section 5.3]{GKM06}.
 \end{rmk}

\begin{thm}\label{gx}
Suppose that $\fs$ is a Cartan subalgebra of type $\gamma$ graded compatible with $x\in \Ao$.  Choose a class representative
$w \in W$ for $\gamma$.  Then, there exists a finite extension $E$ of
$F$ and $g \in G(E)$ satisfying the following properties:
\begin{enumerate}
\item\label{gx1} $g$ is a $w$-diagonalizer of $\fs$;
\item\label{gx2} $g^{-1} \sigma(g) \in N$
is a finite order coset representative for $w \in N / T$;
\item\label{gx3} $-\tau(g) g^{-1} \in (\tx + \fs(0))\cap\ft_\Q$; and
\item\label{gx4} $g = z^{-t} h$, where $t= -\tau(g) g^{-1}\in \ft_\Q$, the image
  $y$ 
  of $t$ in $\Ao$ is a point compatible with $\fs$ (indeed, $y$
  is graded compatible with $\fs$), $h \in G$ simultaneously diagonalizes
  $g^{-1} \sigma (g)$ and $\ft^w$ into $\ft$, and
  $\fs(0)=\Ad(h)(\ft^w)\subset\ft$.
\end{enumerate}
\end{thm}

As explained in Section~\ref{sec:prelim}, the element $z^{-t}$ depends
on a choice of uniformizer for $E$.  However, for any such choice $h=z^t g$ will satisfy
the properties in part~\eqref{gx4}.

\begin{rmk}One may think of the element $g$ described above as a
  generalized Vandermonde matrix.  Consider the case where $G = \GL_n$
  and $\gamma$ is the orbit of the Coxeter element.  Let $\xi$ be a
  primitive $n^{th}$ root of unity and choose $u \in E$ to be an
  $n^{th}$ root of $z$.  Then, it is possible to choose $g$ to be a
  conjugate to the matrix with entries $\left(\xi^{(i-1) j}
    u^{i-1}\right)$.  Here, $t$ is the diagonal matrix with entries
  $(0, 1/n, \ldots, (n-1)/n)$ corresponding to the barycenter of the
  fundamental chamber in $\Ao$.  An explicit embedding of the Coxeter
  torus satisfying the properties in the theorem is given by setting
  $S=F[\varpi]^\times$, where $\varpi=z
  e_{n1}+\sum_{i=1}^{n-1}e_{i(i+1)}$.  More generally, one obtains a
  good embedding of any maximal torus by taking a block-diagonal
  embedding of Coxeter tori.
\end{rmk}

\begin{proof}
  We note that Lemma~\ref{vandermonde} implies that there exists $g\in
  G(E)$ satisfying parts \eqref{gx1} and \eqref{gx3} and such that
  $g^{-1} \s (g)\in N(E)$ is a coset representative for $w \in N/T$.
  Assuming $g=z^{-t}h$ as in part~\eqref{gx4}, $g^{-1} \s
  (g)=\Ad(h^{-1})e^{-2\pi i t}\in N$ has finite order and is
  diagonalized by $h$.

  It remains to prove the rest of part~\eqref{gx4}.  Let $t = -(dg)
  g^{-1} \in \ft(\Q)$.  We first prove that $g = z^{-t} h$ for some $h
  \in G$.  Indeed,
\begin{equation*}
\tau(z^{t} g)  (g^{-1} z^{-t}) = \Ad(z^{t}) \left(t+\tau(g)g^{-1} \right) = 0.
\end{equation*}
It follows that $h = z^{t} g \in G$.  Since $t = - \tau (g)g^{-1} \in
\tx + \fs(0)$, Lemma~\ref{fs} implies that $y$ is graded compatible
with $\fs$.

Finally, suppose that $X\in\ft$.  Computing, we get $(\tau+\ad(\ty))
(\Ad(g) (X))=(\tau+\ad(t)) (\Ad(g) (X))=0$, so $\Ad(g)
(X)\in\fs(E)(0)\subset \ft(E)(0)=\ft$.  In particular, since $z^t\in
T(E)$, $\Ad(h)(X)=\Ad(g)(X)$.  Moreover, $\s( \Ad(g)
(X))=\Ad(g)(w\circ\s(X))=\Ad(g)(wX)$, so $\Ad(g)
(X)\in\fs(0)=(\fs(E)(0))^\Gamma$ if and only if $X\in\ft^w$.  It follows
that $\fs(0)=\Ad(g)(\ft^w)=\Ad(h)(\ft^w)$.
\end{proof}

 We can now identify the set of points in $\Ao$ which are compatible
 with some maximal torus of type $\g$. We denote this set by $\Pi_\g$.
 Note that this set is nonempty by Theorem~\ref{torcon}.  A similar
 problem for $p$-adic classical groups is considered in \cite{BroSt}.

 \begin{cor}\label{strongcompat}
   Let $\g$ denote a conjugacy class in $W$, and choose $x\in \Ao$ and
   $\fs$ of type $\g$ such that $x$ is graded compatible with $\fs$
   and $\tx\in\ft_\Q$.  Then, $\Pi_\g =\{y\in\Ao\mid \ty\in \hW(\tx +
   \fs(0))\}$.   
\end{cor}
Theorem~\ref{compthm} shows the existence of $x$ and $\fs$ as in the
statement of the theorem.  Note that this guarantees that $\tx\in\ft$
for any base field $k$.
\begin{rmk} Theorem~\ref{compthm} shows that $\Pi_\g$ can also be
  characterized as the set  of $y\in\Ao$ graded compatible with some
  conjugate of $\fs$.
\end{rmk}
\begin{proof} Applying Theorem~\ref{compthm}, assume without loss of
  generality that $x$ is graded compatible with $S$.  Lemma~\ref{fs}
  then implies that any point $y$ with $\ty \in \tx + \fs(0)$ is
  (graded) compatible with $S$.  Now, suppose that $v \in \hW$.  We
  may choose a representative $m = z^t n$ for $v$, where $t \in \ft_\Z
  \cong X_*(T)$ and $n \in N$.  Observe that $\tau (m) m^{-1} \in
  \ft_\Q$, so Lemma~\ref{bnorm} implies that $\Ad(m) (\fg_x(i)) =
  \fg_{v x} (i)$.  We deduce that $vx$ is graded compatible with
  $\Ad(m) (S)$.  This shows that $\hW(\tx + \fs(0))\subset\Pi_\g$.

  Now, choose $x'\in\Pi_\g$, and let $\fs'$ be a Cartan subalgebra of
  type $\g$ graded compatible with $x'$.  Choose a representative $w
  \in W$ for $\g$.  By Theorem~\ref{gx}\eqref{gx4}, there exists a
  $w$-diagonalizer $g = z^{-t} h$ (resp. $g' = z^{-t'} h'$) for $S$
  (resp. $S'$) such that $h$ simultaneously diagonalizes $g^{-1}
  \s(g)$ and $\ft^w$ (resp.  $h'$, $(g')^{-1} \s(g')$).  We shall
  write $\zeta = z^{t} \s(z^{-t}) = e^{-2 \pi i t} \in T$ and $\zeta'
  = z^{t'} \s(z^{-t'})$.  We also set $n = h^{-1} \zeta h = g^{-1} \s
  (g) \in N$ and similarly for $n'$.  They are both representatives
  for $w$, so we can take $\tz\in T$ such that $n=n'\tz$.

  Let $S^\flat \subset T$ be the subgroup generated by rational
  cocharacters of $S$.  One may deduce from
  Theorem~\ref{gx}\eqref{gx4} that $h T^{w} h^{-1} = S^\flat$.  We now
  show that $\zeta \in (W\cdot \zeta') S^{\flat}$.  Applying
  Lemma~\ref{Tdecomp}, write $\tz = \de_w(\zeta_1) \zeta_2$, with
  $\zeta_2 \in T^w$.  In particular, $n = (\zeta_1 n' \zeta_1^{-1})
  \zeta_2$.  Moreover,
\begin{equation*}
\begin{aligned}
(h \zeta_1) n' (h \zeta_1)^{-1} &= 
(h \zeta_1) n' \zeta_2 (\zeta_2)^{-1}(h \zeta_1)^{-1} \\
& = (h \zeta_1) n' \zeta_2 (h \zeta_1)^{-1} h \zeta_2^{-1} h^{-1} \in \zeta S^{\flat}.
\end{aligned}
\end{equation*}
The last line follows from the observation above that $s = h
\zeta_2^{-1} h^{-1} \in h T^{w} h^{-1} = S^\flat$.  Thus, $q=h\zeta_1
h'^{-1}$ conjugates $\zeta'$ to $\zeta s$.  We further note that
$qS'^{\flat}q^{-1}=h\zeta_1 T^w \zeta_1^{-1}h^{-1}=h T^w
h^{-1}=S^\flat$.  This implies that $q\zeta' S'^\flat q^{-1}=\zeta
S^\flat$.  The connected centralizer $Z(\zeta' S'^\flat)^0$ is
reductive.  Moreover, it contains the maximal tori $T$ and $q^{-1}Tq$,
so there exists $p\in Z(\zeta' S'^\flat)^0$ such that
$pq^{-1}Tqp^{-1}=T$.  It follows that $qp^{-1}\in N$ and that
conjugation by $qp^{-1}$ and $q$ coincide on $\zeta' S'^\flat$.  If we
let $u\in W$ be the image of $q$, we see that $u (\zeta') = \zeta s$
and $u (S'^\flat) = S^\flat$.

Finally, it is evident that $s = e^{2 \pi i (u(t')-t)}$ and therefore
$t \in u(t') +X_*(T) + \fs(0)$.  Since $\tx \in t + \fs(0)$, $\tx' \in
t' + \fs'(0)$, and $u (\fs'(0)) = \fs(0)$, we conclude that $\tx'$
lies in $\hW(\tx + \fs(0))$.
\end{proof}

\begin{lem}\label{Tdecomp}
Given $w \in W$, let $\de_w : T \to T$ be the homomorphism $\de_w(t) = w^{-1}(t) t^{-1}$.
Then, $T$ has the direct product decomposition $T=\de_w(T) T^w$. 
\end{lem}
\begin{proof}
This follows immediately from the obvious fact that $\Lie(\de_w(T))=\{w^{-1} (X) - X\mid X \in
\ft\}$ and $\ft^w$ are complementary subspaces of $\ft$.
\end{proof}

\section{Regular Strata} \label{sec:RS}

In this section, we introduce a class of strata with regular
semisimple ``leading term''.  We will first do so only for strata
coming from the standard apartment.

If $\tb\in\hfg^\vee$, we let $Z(\tb)$ (resp. $Z^0(\tb)$) denote the
stabilizer (res. connected stabilizer) of
$\tb$ in $\hG$ under the coadjoint action.  The corresponding Lie
algebra $\fz(\tb)$ is the stabilizer of the Lie algebra action.

\begin{defn}
  A stratum $(x, r, \b)$ with $x\in\Ao$ is \emph{graded regular} if
  $Z^0(\tbo)$ is a maximal torus which is compatible with $x$.
  We call $Z^0(\tbo)$ the connected centralizer of the stratum.
\end{defn}
It is of course equivalent to check that the Lie centralizer
$\fz(\tbo)$ of the stratum is a Cartan subalgebra compatible with $x$.
It is obvious that a graded regular stratum is fundamental.

Given a torus $S$, we write $\rho_\fs : \hfg^\vee \to \fs^\vee$ for
the restriction map.  If $\fs$ is compatible with $x\in\B$,
$\rho_\fs(\hfg^\vee_{x,r}))\subset\fs^\vee_r$ and
$\rho_\fs(\hfg^\vee_{x,r+}))\subset\fs^\vee_{r+}$ for all
$r$. Moreover, if $\fs$ is graded compatible with $x\in\Ao$, then $\rho_\fs(\hfg^\vee_{x}(r))\subset\fs^\vee(r)$.

\begin{lem}[Tame Corestriction]\label{cores} 
Take $x\in \Ao$, and let $(x, r, \b)$ be a graded regular stratum with
connected centralizer $S$.  Then, there is a morphism of $\fs$-modules $\pi_\fs : \hfg \to \fs$ satisfying the following properties:
\begin{enumerate}
\item\label{cores1} $\pi_\fs$ restricts to the identity on $\fs$;
\item\label{cores2} $\pi_\fs(\hfg_{x,\ell}) = \fs_\ell$ and 
$\pi_\fs^*(\fs_{\ell}^\vee) \subset \hfg_{x, \ell}^\vee$;
\item\label{cores3} the kernel of the restriction map
\begin{equation*}
\brho_{\fs,\ell} : (\pi_\fs^*(\fs^\vee )+ \hfg^\vee_{x,\ell-r}) / \hfg^\vee_{x,(\ell-r)+} \to 
\fs^\vee / \fs^\vee_{(\ell-r)+}
\end{equation*}
is given by the image of $\ad^*(\hfg_{x,\ell}) (\tb)$ modulo
$\hfg^\vee_{x, (\ell-r)+}$, where $\tb\in \hfg^\vee_{x, -r}$ is any
representative of $\b$;
\item\label{cores4} if $Z \in \fs$ and $X \in \hfg$, then 
$\left< Z, X \right>_\ddz = \left<Z, \pi_\fs (X) \right>_\ddz$;
\item\label{cores5} $\pi_\fs$ (resp. $\pi_\fs^*$) commutes with the
  adjoint action of the normalizer $N(S)$ of $S$; and
\item\label{cores6} the image $\pi_\fs^*(\fs_\ell^\vee) $ consists
of those elements in $\hfg_{x, \ell}^\vee$ stabilized by
$S$.
\end{enumerate}
\end{lem}

\begin{rmk}\label{tamermk} The proof will actually show that $\ad^*(\hfg_{x,\ell}) (X)$ modulo
  $\hfg^\vee_{x, (\ell-r)+}$ is in the kernel of $\brho_{\fs,\ell}$
  for any $X\in\pi^*_\fs(\fs^\vee_{-r})+\hfg^\vee_{x,-r+}$.
\end{rmk}
\begin{rmk} We will usually omit the subscript $\ell$ on
  $\brho_{\fs,\ell}$ when it is clear from context.
\end{rmk}
\begin{proof}
  Recall that $\pi_\fs$ is the orthogonal projection of $\hfg$ onto
  $\fs$ with respect to the $F$-bilinear invariant pairing $\left<,
  \right>$ on $\hfg$ obtained by extending scalars.  The invariance of
  the form immediately shows that $\pi_\fs$ is an $\fs$-module map
  that is equivariant with respect to the adjoint action of
  $N(S)$. The first and fourth statements are trivial.

Recall that $F$-duals and smooth $k$-duals of $F$-vector spaces are
identified using the map $\kappa$ described before
Proposition~\ref{dual}.  We will use this identification and the
results of this proposition in the rest of the proof without comment.

Since $x$ is compatible with $\fs$, we have $\fs_\ell \subset
\hfg_{x,\ell}$, so $\pi_\fs(\hfg_{x,\ell}) \supset \fs_\ell$.  Now,
suppose that $Z \in \fs_{-\ell+}$ and $X \in \hfg_{x,\ell}$.  In order
to show that $\pi_{\fs} (X) \in \fs_\ell,$ it suffices to show that
$\left< X, Z \right>_\ddz = 0$.  This follows immediately, since $Z
\in \hfg_{x, -\ell+}=\hfg_{x, \ell}^\perp$.  Also, if
$\a\in\fs_\ell^\vee$, then $\pi_\fs^*(\a)(\hfg_{x,
  \ell+})=\a(\pi(\hfg_{x, \ell+}))=\a(\fs^\vee_{\ell+})=0$, so
$\pi^*_\fs(\fs_\ell^\vee)\subset \hfg_{x, \ell}^\vee$.  This proves
the second statement.

Suppose that $\a\in \fs^\vee$ and $s \in S$.  Then, if $X \in \hfg$,
$\Ad^*(s) (\pi_\fs^*(\a)) (X) = \pi_\fs^*(\a) (\Ad(s^{-1}) (X)) = \a
(\pi_\fs(\Ad(s^{-1}) (X)))$.  The right hand side is equal to
$\pi_\fs^* (\a) (X)$ by part~\eqref{cores5}.  Thus,
$\pi_\fs^*(\fs^\vee)$ is stabilized by $S$.  On the other hand, if $\a
\in \hfg^\vee$ is stabilized by $S$, then, writing $\a=\langle
Y,\cdot\rangle$, one sees that $\Ad(s)(Y)=Y$ for all $s\in S$, i.e.,
$Y\in\fs$.  By part~\eqref{cores4}, $\a(X) = \a(\pi_\fs(X))$, so $\a
\in \pi_\fs^*(\fs)$.  This proves \eqref{cores6}.

Finally, we consider the third statement.  The image of
$\ad^*(\hfg_{x,j})(\tb)$ modulo $\hfg^\vee_{x, (j-r)+}$ is independent
of the choice of representative, so we can take
$\tb=\tbo\in\hfg^\vee_x(-r)$.  Since $S$ stabilizes $\tbo$,
$\ad^*(Z)(\tbo)=0$ for $Z\in\fs$, and it follows that for any
$X\in\hfg$, $\ad^*(X)(\tbo)(Z)=-\ad^*(Z)(\tbo)(X)=0$, i.e.,
$\rho_\fs(\ad^*(\hfg)(\tbo))=0$.  Therefore, $(\ad^*(\hfg_{x,
  j})(\tbo) + \hfg^\vee_{x, (j-r)+})/\hfg^\vee_{x, (j-r)+}) \subset
\ker(\brho_{\fs,j})$. For convenience, we denote the former space by
$C_j$.

We will show that $\dim C_j=\dim\ker(\brho_{\fs,j})$.  Observe that
$\rho_\fs\circ \pi^*_\fs$ is the identity on $\fs^\vee$, so
$\ker(\brho_{\fs,j})=\ker(\hfg^\vee_{x,j-r}/\hfg^\vee_{x,(j-r)+}\to\fs^\vee_{j-r}/\fs^\vee_{(j-r)+})$.
This second map is surjective, so
$\dim\ker(\brho_{\fs,j})=\dim\hfg^\vee_x(j-r)-\dim\fs^\vee(j-r)$.
Next, suppose that $X \in \hfg_{x, j}$ and $\ad^* (X) (\tbo) \in
\hfg_{x, (j-r)+}$.  Since this is always the case if
$X\in\hfg_{x,j+}$, we can assume that $X\in\hfg_x(j)$.  By
Proposition~\ref{ev}, $\ad^*(X) (\tbo) \in \hfg^\vee_{x} (j-r)\cap
\hfg^\vee_{x, (j-r)+} = \{0\}$, so $X\in\fs\cap\hfg_x(j)=\fs(j)$.  We
thus obtain $\dim(C_j)=\dim\hfg_x(j)-\dim\fs(j)$.

We now sum these two dimensions over $j$ in the full period $\ell\le j
<\ell+1$.  Note that summing $j$ over the full period $\ell\le j
<\ell+1$ gives $\sum_{\ell\le
  j<\ell+1}\dim\hfg_x(j)=\dim\fg=\dim\fg^\vee=\sum_{\ell\le
  j<\ell+1}\dim\hfg^\vee_x(j-r)$ and similarly, $\sum_{\ell\le
  j<\ell+1}\dim\fs_x(j)=\sum_{\ell\le j<\ell+1}\dim\fs^\vee_x(j)$.
(Of course, both sides are zero if $\hfg^\vee_x(j-r)=\{0\}$.)  Thus,
$\sum_{\ell\le j<\ell+1} \dim C_j=\sum_{\ell\le
  j<\ell+1}\dim\ker(\brho_{\fs,j})$.  Since each term on the right is
greater or equal to the corresponding term on the left, we get $\dim
C_j=\dim\ker(\brho_{\fs,j})$ for all $j$.  In particular,
$C_\ell=\ker(\brho_{\fs,\ell})$.
\end{proof}

\begin{prop}\label{r=0}
  If $(x, 0, \b)$ is a graded regular stratum with $x\in\Ao$, then
  there exists $m \in \hG_x$ such that $m \cdot (x, 0, \b) = (x, 0, \b')$
  is regular with connected centralizer $\hT$.   The stratum $(x, 0, \b')$
  is uniquely determined by $(x, 0, \b)$ up to the action of
  $(\hN \cap \hG_x)/ \hT_x \cong (N \cap H_x)/T$.
  
\end{prop}

\begin{proof}  
  Write $S = Z^0(\tbo)$, so $\ft'=(d\theta_x)^{-1}(\fs_0/\fs_{0+})$ is
  a Cartan subalgebra in $\fh_x$.  Choose $h\in H_x$ such that
  $\Ad(h)\ft'=\ft$, and let $m\in\hG_x$ be a lift of $\theta_x(h)$.
  Define a new stratum $(x,0,\b')=m\cdot(x,0,\b)$ which has
  representative $\tbo' \in (\Ad^*(m)(\tbo) + \hfg^\vee_{x +})\cap
  \hfg^\vee_x(0)$.  It suffices to show that $Z^0 (\tbo') =\hT$.
  
  Let $Z_0 (\tbo') = Z^0 (\tbo') \cap \hG_{x}$ and $Z_+(\tbo') = Z^0
  (\tbo') \cap \hG_{x +}$.  Observe that $\theta_x^{-1} (Z_0 (\tbo')/
  Z_+ (\tbo')) = T$, so by the $T$-equivariance of $\theta_x$, if $t
  \in T \subset \hG_{x}$, $\Ad^*(t) (\tbo') - \tbo' \in \hfg^\vee_{x
    +} \cap \hfg^\vee_x (0) = \{0\}$.  We deduce that $\hT \subset Z^0
  (\tbo')$.
  
  On the other hand, the $T$-equivariance of $\theta_x$ also implies
  that $Z_0 (\tbo') \subset T \hG_{x, +}$.  Suppose $g \in Z_0 (\tbo')
  \cap \hG_{x, \ell}$ for $\ell > 0$.  Write $g \in \exp (X)\hG_{x,
    \ell+}$ for some $X \in \hfg_x(\ell)$.  Then, $\tbo' = \Ad^*(g)
  (\tbo') \in \tbo' + \ad^*(X)(\tbo') + \hfg_{x, \ell+}$.  It follows
  that $\ad^*(X) (\tbo') = 0$.  Therefore, $\ad^*(\Ad(m^{-1}) (X))
  (\tbo) \in \hfg^\vee_{x, \ell+}$.  Finally, this implies that
  $\Ad(m^{-1}) (X) \in \fs_\ell + \hfg_{x, \ell+}$ and $X \in
  \hft_\ell + \hfg_{x, \ell+}$.  Thus, $Z_0 (\tbo') \cap \hG_{x, \ell}
  \subset \hT_\ell \hG_{x, \ell+}$.  A standard limit argument shows
  that $Z_0(\tbo') \subset \hT_0$ and thus $Z^0 (\tbo') \subset \hT$.

  Now suppose that $(x, 0, \b)$ is a graded regular stratum with
  connected centralizer by $\hT$, and there exists $m \in \hG_x$ such
  that the same holds for $m \cdot (x, 0, \b)$.  It suffices to show
  that there exists $n \in \hN \cap \hG_x$ such that $n \cdot (x, 0,
  \b) = m \cdot (x, 0, \b)$.  Recall that $m' \cdot (x, 0, \b) = m
  \cdot (x, 0, \b)$ whenever $m' \in m \hG_{x+}$.  Let $\bar{m}$ be
  the image of $m$ in $H_x$ under the composition $\hG_x \to \hG_x /
  \hG_{x, +} \xrightarrow{\theta^{-1}_x} H_x$.  Then,
  $\theta'_x(\bar{m}) \in m \hG_{x,+}$ and $\Ad^*(\theta'_x (\bar{m}))
  (\tbo) \in \hfg^\vee_x (0)$.  Since $\tbo$ is a regular element of
  $\hfg_x(0)$ stabilized by $\hT$, and the same is true for
  $\Ad^*(\theta'_x (\bar{m})) (\tbo)$, we deduce that
  $\theta'_x(\bar{m}) \in \hN \cap \hG_x$.

Finally, it is clear that the action of $\hT_x$ fixes $(x, 0, \b)$.  Moreover,
the correspondence $m \mapsto \bar{m}$ above determines a homomorphism
from $\hN \cap \hG_x \to N \cap H_x$.  It is easily checked that this homomorphism
induces an isomorphism $(\hN \cap \hG_x)/ \hT_x \cong (N \cap H_x)/ T$.
\end{proof}

If $(x,0,\b)$ is graded regular with centralizer $\hT$, then, since
$x$ is graded compatible with $\hT$, $\tbo\in \ft\ddz$ and
$\Res(\tbo)$ is a regular element of $\ft$.

\begin{prop}\label{resequiv}
  Suppose that $(x, 0, \b)$ is a graded regular stratum with connected
  centralizer $\hT$, and let $m \in \hN \cap \hG_x$ determine an
  element $w \in W$.  Then, for any root $\a$, $\a(\Res(\tbo)) +
  \a(\tx) \in \Z_{< \a (\tx)}$ if and only if $w(\a)(\Res(\Ad^*(m)
  (\tbo)) + w(\a) (\tx) \in \Z_{<w(\a)(\tx)}$.
\end{prop}
\begin{proof}
  It is evident that $w(\a) (\Res(\Ad^*(m) (\tbo)) = w(\a) (w
  \Res(\tbo)) = \a(\Res(\tbo))$.  It is now obvious that
  $\a(\Res(\tbo)) + \a(\tx) < \a (\tx)$ if and only if
  $w(\a)(\Res(\Ad^*(m) (\tbo)) + w(\a) (\tx) <w(\a)(\tx)$, so it
  remains to show that if one of the expressions on the left is an
  integer, the other is as well.  In particular, it suffices to show
  that $\a(\tx) - w (\a) (\tx) \in \Z$ for all $\a$, or equivalently
  (replacing $\a$ by $w^{-1}\a$), $\a(w \tx - \tx) \in \Z$.

  Exponentiating, it is enough to show that $\Ad(\exp(2 \pi i (w \tx -
  \tx))) (X) = X$.  However, by Proposition~\ref{r=0}, there is an
  element $h \in H_x \cap N$ such that $w \tx = \Ad(h) (\tx)$.
  Therefore, $\exp(2 \pi i(w \tx - \tx)) = h \exp(2 \pi i \tx) h^{-1}
  \exp(2 \pi i \tx)^{-1}$.  Since $h$ commutes with $\exp(2 \pi i
  \tx)$, this is the identity.
\end{proof}
\begin{defn}\label{nonres}
  Let $(x, 0, \b)$ be a graded regular stratum with $x \in \Ao$.
  Without loss of generality, assume that $\hT = Z^0 (\tbo)$.  We say
  that $(x, 0, \b)$ is \emph{resonant} if for some root $\a$ of $\hT$,
\begin{equation*}
\a (\Res (\tbo)) + \a(\tx) \in \Z_{<\a(\tx)}.
\end{equation*}
\end{defn}
By Proposition~\ref{r=0}, one can always chose $m \in \hG_x$ such that
$m \cdot (x, 0, \b) = (x, 0, \b')$ has connected centralizer $\hT$.
 The only ambiguity in the choice of $m$ lies in the fact that, for
any $n \in \hN \cap \hG_x$, $n m$ will also ``diagonalize'' $(x, 0,
\b)$.  However, Proposition~\ref{resequiv} implies that the condition
in the definition above holds for $\tbo$ if and only if it holds for
$\Ad^*(n) (\tbo)$.

\begin{defn}
  A stratum $(x,r,\b)$ is called \emph{regular} if the following
  conditions are satisfied: all subgroups $Z^0(\tb)$ for any
  representative $\tb$ are $\hG_{x+}$-conjugate maximal tori
  compatible with $x$, and additionally the stratum must be
  nonresonant when it has depth $0$.  If $\S$ is the
  $\hG_{x+}$-conjugacy class of maximal tori determined by these
  stabilizers, we say the stratum is \emph{$\S$-regular} (or even
  $S$-regular, if $S\in\S$).
\end{defn}

Note that any $\hG_{x+}$-conjugate of a representative $\tb$ is also a
representative, so the set of connected centralizers of a regular stratum is a
full $\hG_{x+}$-orbit.  Also, the compatibility with $x$ need only be
checked on a single $S\in\S$.  The following proposition shows that
$(x,r,\b)$ is regular if and only if it is conjugate to a graded
regular stratum coming from $x'\in\Ao$.

\begin{prop}\label{leadingterm}\mbox{} 
\begin{enumerate}\item If $x\in\Ao$, then the stratum $(x,r,\b)$ is
  graded regular (and nonresonant if $r = 0$) if and only if it is regular.
\item \label{leadingterm2} Let $(x,r,\b)$ be an $S$-regular stratum.
  If $\tb \in \pi_\fs^*(\fs^\vee_{-r}) + \hfg^\vee_{x, \ell-r}$ for
  $\ell > 0$ is any representative, then there exists $p \in \hG_{x,
    \ell}$ such that $\Ad^*(p) (\tb) \in \pi^*_\fs (\fs_{-r}^\vee)$.
\end{enumerate}
\end{prop}
\begin{proof}
  It is trivial that regular implies graded regular.  Now, assume that
  $(x,r,\b)$ is graded regular, and set $S=Z^0(\tbo)$.  First, suppose
  that $\tb\in\pi_\fs^*(\fs^\vee)$ is a representative.  We show that
  $\fz(\tb)=\fs$.  By Lemma~\ref{cores}\eqref{cores6},
  $\fs\subset\fz(\tb)$.  Conversely, suppose this inclusion is strict.
  Then, there exists an element of $\fz(\tb)$ of the form $Y+Y'$ with
  $Y\in\hfg_x(\ell)\setminus\fs(\ell)$ and
  $Y'\in\hfg_{x,\ell+}\setminus\hfg_{x,\ell}$.  Since $\tb=\tbo+\a$
  with $\a\in\hfg^\vee_{x,-r+}$, $\ad^*(Y+Y')(\tb)$ equals
  $\ad^*(Y)(\tbo)\in\hfg^\vee(\ell-r)$ plus higher order terms.  It
  follows that $\ad^*(Y)(\tbo)=0$, contradicting that fact that
  $Y\notin\fs$.

  Next, let $\tb$ be any representative.  It suffices to find $p \in
  \hG_{x, 0+}$ such that $\Ad^*(p) (\tb) \in \pi_\fs^*(\fs^\vee)$, as
  the previous argument will then imply that
  $\fz(\tb)=\Ad(p^{-1})(\fs)$.  Note that $\tb \in
  \pi_\fs^*(\fs_{-r}^\vee) + \hfg_{x, (\ell-r)+}^\vee$ for some $\ell
  >0$.  Since $\tb - \pi_\fs^*(\rho_\fs (\tb)) + \hfg_{x,
    (\ell-r)+}^\vee \in \ker(\brho_\fs)$,
  Lemma~\ref{cores}\eqref{cores3} states that there exists $X \in
  \hfg_{x, \ell}$ such that $\ad^*(X) (\tb) \in \tb -
  \pi_\fs^*(\rho_\fs (\tb)) + \hfg_{x, (\ell-r)+}^\vee$.  Take $p_\ell
  = \exp(-X) \in \hG_{x, \ell}$.  Then, $\Ad^*(p_\ell) (\tb) \in
  \pi_\fs^*(\fs_{-r}^\vee) + \hfg_{x, (\ell-r)+}^\vee$.  Recursively
  applying this argument, we obtain an increasing sequence $\ell_i$
  (with $\ell_1=\ell$) and elements $p_{\ell_i}\in\hG_{x,\ell_i}$ such
  that
  $\Ad^*(p_{\ell_m})\cdots\Ad^*(p_{\ell_1})(\tb)\in\pi_\fs^*(\fs_{-r}^\vee)
  + \hfg_{x, (\ell_m-r)+}^\vee$ for all $m$.  Setting
  $p=\lim(p_{\ell_m}\cdots p_{\ell_1})\in\hG_{x,\ell}$, we have
  $\Ad^*(p)(\tb)\in \pi_{\fs}^*(\fs_{-r}^\vee)$ as desired.  This also
  proves the second statement for $x\in\Ao$.

Finally, the general case of part~\eqref{leadingterm2} follows by
conjugating a regular stratum $(x,r,\b)$ to $(x',r,\b')$ with
$x'\in\Ao$.  One need only observe that conjugation preserves Moy-Prasad
filtrations while if $S'=gSg^{-1}$, then
$\pi_{\fs'}\Ad(g)=\Ad(g)\pi_\fs$.
\end{proof}

Not every maximal torus can be the connected centralizer of a regular stratum.
Recall that a Weyl group element $w$ is called regular if it has an
eigenvector in the reflection representation whose stabilizer is
trivial~\cite{Spr74}.  Equivalently, it has a regular semisimple
eigenvector in $\ft$.  Note that the eigenvalue of such a regular
eigenvector can equal one only for the identity element of $W$.  We
say that a maximal torus (or Cartan subalgebra) has \emph{regular type} if it has
type $\g$, with $\g$ a regular conjugacy class in $W$.

\begin{cor}\label{regtorus}
  Let $\g$ be a conjugacy class in $W$.  Then, there exists an
  $S$-regular stratum for some maximal torus $S$ of type $\g$ if and
  only if $\g$ is regular. In this case, the set of $x\in\Ao$ which
  support an $S$-regular stratum for some $S$ of type $\g$ is
  precisely $\Pi_\g$, which is described explicitly in
  Corollary~\ref{strongcompat}.
\end{cor}
\begin{proof}
  It is obvious that the split torus $\hT$, corresponding to $e\in W$,
  admits regular strata (with $r=0)$, so we assume that $S$ is
  nonsplit.  Let $E/F$ be a degree $e$ extension over which $S$
  splits, and let $g\in G(E)$ be a $w$-diagonalizer of $S$.  It
  follows that $\a=\Ad^*(g^{-1}) (\tb) \in \hft^\vee_E (-e r)$ has
  connected centralizer $\hT$.  One can then find $v\in z^{-r}\ft$
  regular semisimple such that $\a=\langle v,\cdot\rangle$.  With $\s$
  our fixed generator for $\Gal(E/F)$ and $\xi$ the $e$-th root of
  unity defined by $\s(z^{1/e})=\xi z^{1/e}$, we have
  $w^{-1}(v)=\s(v)=\xi^{-re}v$.  We deduce that $z^r v\in\ft$ is a
  regular eigenvector for $w$ (with eigenvalue $\xi^{re}$).

  If $\g$ is a nonidentity regular conjugacy class, $e^{2\pi i r}(\ne
  1)$ is a regular eigenvalue for $\g$ with $r>0$, and $x\in\Pi$, it
  is easy to write down an explicit regular stratum based at $x$ of
  depth $r$ whose connected centralizers are of type $\g$.  A
  classification of such strata is given in Theorem~\ref{isothm}.
\end{proof}

\begin{exam} For $\GL_n$, a conjugacy classes in the Weyl group $S_n$ is
  regular if its cycle decomposition consists of either $n/k$
  $k$-cycles or $(n-1)/k$ $k$-cycles and an additional $1$-cycle.  The
  corresponding maximal tori are the isomorphic to
  $(F[z^{1/k}]^\times)^{n/k}$ and $(F[z^{1/k}]^\times)^{(n-1)/k}\times
  F^\times$ respectively.  A maximal torus of the former type is
  called uniform; formal connections contains $S$-regular strata for
  $S$ uniform were studied in ~\cite{BrSa1}.  
\end{exam}

\section{Isomorphism classes of flat $G$-bundles}\label{isosect}

In this section, we give a parameterization for the space of
isomorphism classes of formal flat $G$-bundles that contain regular
strata.  More precisely, let $S$ be a maximal torus of regular type
which is compatible with some point in $\Ao$.  If $r\ge 0$, let
$\cC(S,r)$ be the category of connections of slope $r$ that contain an
$S$-regular stratum.  For each $x\in\Ao$ compatible with $x$, we also
define the category $\cCfr_x(S,r)$ of \emph{framed} flat formal
$G$-bundles whose objects are quadruples $\cF=(\sG,\n,(x,r,\b),\phi)$,
where $(\sG,\n)$ is an object in $\cC(S,r)$ containing the regular
stratum $(x,r,\b)$ with respect to the trivialization $\phi$.  The
morphisms in $\Hom(\cF,\cF')$ consists of isomorphisms
$\psi:(\sG,\n)\to(\sG',\n')$ such that
$\psi'=\phi'\circ\psi\circ\phi^{-1}\in \hG_x$ and $\psi'^*(\b')=\b$.
There is a forgetful ``deframing'' functor $\cCfr_x(S,r)\to\cC(S,r)$.
We show that the moduli space of $\cCfr_x(S,r)$ is the space of
\emph{$(S,r)$-formal types}--an explicitly-determined
open\footnote{for $r>0$} subset of $\fs^\vee_{-r}/\fs^\vee_{0+}$.  The
space of formal types is endowed with an action of the relative affine
Weyl group, and the moduli space of $\cC(S,r)$ is the corresponding
orbit space.

Throughout this section, $S$ will denote a maximal torus of regular
type compatible with some point in $\Ao$.  For clarity of exposition,
we will assume that $S=\hT$ when $r=0$.  This restriction is
unnecessary, but it allows one to avoid the notational complications
inherent in discussing the resonance condition for other split maximal
tori.

\subsection{Framed flat $G$-bundles and formal types}

In this section, we show that the category $\cCfr_x(S,r)$ is
essentially independent of the choice of $x$ for $r>0$ and compute its
moduli space.

\begin{thm}\label{diagthm}
  Suppose that $(\sG, \n)$ contains an $S$-regular stratum $(x, r,
  \b)$ with respect to the trivialization $\phi$.
\begin{enumerate}
\item\label{diag1} There exists $p \in \hG_{x+}$ and an element $\tA
  \in \pi_\fs^*(\fs_{-r}^\vee)$ such that $[\n]_{p \phi}- \tx \ddz =
  \tA - \tx \ddz$.
\item\label{diag2} The orbit of $\tA- \tx \ddz$ under
  $\hG_{x+}$-\emph{gauge} transformations contains $\tA - \tx \ddz +
  \hfg^\vee_{x+}$.
\item\label{diag3} If $\tA^1, \tA^2 \in \pi^*_\fs(\fs_{-r}^\vee)$ both
  determine the same regular stratum $(x, r, \b)$, and there exists $p
  \in \hG_{x}$ such that $p \cdot \tA^1 \in \tA^2 + \hfg^\vee_{x+}$,
  then $\tA^1 \in \tA^2 + \pi^*_\fs(\fs^\vee_{0+})$.
\end{enumerate}
\end{thm}

Before proving this theorem, we show how it can be recast in the
language of \emph{formal types}.  Since $\pi_\fs^* : \fs^\vee \to
\hfg^\vee$ is an injection which is compatible with the Moy-Prasad
filtration induced by $x$, the above theorem suggests that one may
parameterize flat $G$-bundles of slope $r$ containing $S$-regular
strata by elements of $\fs^\vee_{-r} / \fs^\vee_{0+}$.  In the
following, we will adopt the notational convention that whenever $\tA
\in \pi^*_\fs (\fs_{-r}^\vee)$, then $A = \rho_\fs (\tA) +
\fs^\vee_{0+} \in \fs^\vee_{-r} / \fs^\vee_{0+}$.  Similarly if $A \in
\fs_{-r}^\vee/ \fs^\vee_{0+}$, then $\tA \in
\pi^*_\fs(\fs^\vee_{-r})$, unless already defined, will denote an
arbitrary element of $\pi^*_\fs(A)$.

\begin{defn}\label{ftdef}
  A functional $A \in (\fs_0)^\vee\cong \fs^\vee / \fs^\vee_{0+}$ is
  called an \emph{$S$-formal type of depth $r$} if
\begin{enumerate}
\item the smallest congruence ideal contained in $A^\perp$
is $\fs_{r+}$; 
\item there exists $x \in \Ao$ compatible with $S$; and
\item \label{ft3} for some $x$ compatible with
  $S$, the corresponding stratum $(x, r, A^x)$ is
  $S$-regular, where $A^x$ is the functional induced by
  $\tA-\tx \ddz$.
\end{enumerate}
We denote the space of $S$-formal types of depth $r$ by $\AT(S,r)$,
which we will view as a subset of $\fs^\vee_{-r}/\fs^\vee_{0+}\cong
(\fs_0/\fs_{r+})^\vee$.  An $S$-formal type is any element of
$\cup\AT(S,r)$.
\end{defn} 
\begin{rmk} When $r>0$, $\AT(S,r)$ is an open subset of the affine
  space
  $\fs^\vee_{-r}/\fs^\vee_{0+}\cong\bigoplus_{j=-r}^0\fs^\vee(j)$ with
  the summation only including $j\in\Crit(\fs)$.  Indeed, a coset
  $\fs_{-r}^\vee/\fs_{0+}^\vee$ corresponds to an element of
  $\AT(S,r)$ if and only if its projection onto $\fs^\vee(-r)$ is
  regular, which is clearly an open condition.  If $r=0$, the proof of
  Proposition~\ref{ftlem} below shows that $\AT(\hT,0)\cong\{X\in
  \ft\mid \a(X)\notin \Z \text{ for all }\a\in\Phi\}$.  This is not
  Zariski-open, but if $k=\C$, it is open in the complex topology.
\end{rmk}

\begin{prop}\label{ftlem}
  Suppose that $A \in (\fs_0)^\vee$ is an $S$-formal type of depth
  $r$.
\begin{enumerate}\item If $r>0$, $(x,r,A^x)$ is
  $S$-regular for all $x\in\Ao$ compatible with $S$. 
\item When $r = 0$, the stratum $(x, r, A^x)$ is $\hT$-regular if and
  only if $\a(\Res(\tA)) \ne \a(\tilde{x})$ for all
  $\a\in\Phi$. In
  particular, if $H_x=G$, then $(x, r, A^x)$ is $\hT$-regular for all
  $A\in\AT(\hT,0)$.
\end{enumerate}
\end{prop}

\begin{proof}  Choose $y\in\Ao$ such that $(y,r,A^y)$ is $S$-regular.
  When $r > 0$, the functional $A^x$ is induced by
  $\pi_\fs^*(A)$, so it is immediate that $(x,r,A^x)$ is $S$-regular for any $x \in \Ao$
  compatible with $S$.

  Now, suppose that $r = 0$ and $S=\hT$.  The nonresonance condition
  states that for all $\a \in \Phi$, $\a(\Res(\tA))=\a(\Res(\tA - \ty
  \ddz)) + \a(\ty) \notin \Z_{<\a(\ty)}$.  Since this condition holds
  for $-\a$, either $\a(\Res(\tA))$ is not an integer or
  $\a(\Res(\tA)) = \a(\ty)$.  However, if $\a(\Res(\tA)) = \a(\ty)$,
  then $\a(\Res(\tA - \ty \ddz)) = 0$ and $(y, 0, A^y)$ is not a
  regular stratum.  We deduce that $\a(\Res(\tA)) \notin \Z$ for all
  $\a\in \Phi$.

  The set of $x \in \Ao$ such that $\tA - \tilde{x} \ddz$ does not
  determine a regular stratum thus consists of those points for which
  there exists $\a \in \Phi$ such that $\a(\Res(\tA)) =
  \a(\tilde{x})$.  This is a finite union of hyperplanes.  (Since
  $\tx\in\ft_\R$, this condition is vacuous for roots $\a$ with
  $\a(\Res(\tA))\notin\R$.)  Finally, suppose that $H_x=G$.  If the
  corresponding stratum $(x, r, A^x)$ is not regular, then there
  exists $\a$ such that $\a(\Res(\tA)) = \a(\tx)$.  Since $H_x=G$,
  $\a(\tx) \in \Z$.  This contradicts the fact that $\a(\Res(\tA))
  \notin \Z$.  On the other hand, if $H_x\ne G$, there exists $\a$ for
  which $\a(\tx) \notin \Z$.  We can thus choose $A\in\AT(\hT,0)$ such
  that $\a(\Res(\tA)) =\a(\tx)$.
\end{proof}

\begin{cor} If $x$ and $y$ are both compatible with $S$ and $r>0$,
  then the categories $\cCfr_x(S,r)$ and $\cCfr_y(S,r)$ are
  canonically isomorphic.
\end{cor}
\begin{proof} If $A$ is the formal type for
  $\cF_x=(\sG,\n,(x,r,\b^x),\phi)$, then the proposition shows that
  the functor $\cF_x\mapsto \cF_y=(\sG,\n,(y,r,\b^y),\phi)$ is the
  desired isomorphism.
\end{proof}

We can now describe the moduli spaces of these categories of framed
flat $G$-bundles.  Let $\AT_x(\hT,0)$ be the Zariski-open subset of
$\AT(\hT,0)$ consisting of those $A$ satisfying
$\a(\Res(\tA))\ne\a(\tx)$ for all $\a\in\Phi$. It is clear that
$\AT_x(\hT,0)=\AT(\hT,0)$ if and only if $\a(\tx)\in\Z$ for all
$\a\in\Phi$; equivalently, this holds precisely when $H_x=G$.

\begin{thm}\mbox{}   \begin{enumerate} \item If $r>0$, $\AT(S,r)$ is the moduli space of
    $\cCfr_x(S,r)$.
  \item The moduli space of $\cCfr_x(\hT,0)$ is a subset of
    $\AT(\hT,0)$, with equality if and only if $H_x=G$.
\end{enumerate}
\end{thm}
\begin{proof}

  If $(\sG, \n)$ contains an $S$-regular stratum $(x, r, \b)$, then
  Theorem~\ref{diagthm}\eqref{diag1} shows that there is an element
  $\tA\in\pi^*_\fs(\fs^\vee_{-r})$ and a trivialization $\phi$ of $\n$
  such that $\tA-\tx\ddz$ determines $\b$ and $[\n]_\phi |_{\hfg_{x,
      0}} = \tA$.  Thus, any isomorphism class can be represented by a
  quadruple $\cF=(\sG,\n,(x,r,\b),\phi)$ satisfying these properties.
  We define a map from the moduli space to $\AT(S,r)$, sending the
  class of $\cF$ to the formal type $A=\rho_\fs(\tA)+\fs^\vee_{0+}$.
  Part~\eqref{diag3} shows that this map is well-defined.  Moreover,
  the map is injective: if $\cF'$ is another such good representative
  of an isomorphism class for which $\tA'$ determines the same formal
  type $A$, then part ~\eqref{diag2} shows that $\cF'\cong \cF$.
  Thus, the moduli space of $\cCfr_x(S,r)$ is a subset of $\AT(S,r)$.
  One easily sees that the image lies in $\AT_x(\hT,0)$ when $r=0$.

  It remains to compute the image of the formal types map.  Let $\sG^\triv$ be the trivial principal formal $G$-bundle. If $A\in
  \AT(S,r)$, define a flat structure on $\sG^\triv$ via
  $[\n_{\tA}]_{\id}=\tA$, where $\id$ is the identity trivialization.  The flat
  $G$-bundle $(\sG^\triv,\n_{\tA})$ contains the stratum $(x,r,A^x)$ with
  respect to $\id$.  When $r>0$, the class of
  $(\sG^\triv,\n_{\tA},(x,r,A^x),\id)$ is mapped to $A$, so the formal
  types map is surjective.  If $r=0$, the elements of $\AT_x(\hT,0)$ are
  precisely those formal types for which $(x,r,A^x)$ is a regular
  stratum, and we see that the image is $\AT_x(\hT,0)$ in the same
  way.
\end{proof}

We now turn to the proof of Theorem~\ref{diagthm}.  We begin with two
lemmas.

\begin{lem}\label{redstep}
Suppose that $(\sG, \n)$ contains an $S$-regular stratum $(x, r, \b)$
with respect to the trivialization $\phi$, and assume that 
$[\n]_\phi  \in \pi_\fs^*(\fs^\vee) + \hfg^\vee_{x, \ell-r}$ for some $\ell > 0$.  
\begin{enumerate}
\item\label{redstep1} There exists $p \in \hG_{x, \ell}$ such that 
$[\n]_{p \phi}  \in \pi_\fs^* \circ \rho_\fs ([\n]_\phi) + \hfg^\vee_{x, (\ell-r) +}$.  
\item\label{redstep2} If $q \in \hG_{x, \ell}$ and $r > 0$, then 
\begin{equation*}
\rho_\fs([\n]_{q \phi}) \in \rho_\fs ([\n]_{\phi}) + \fs^\vee_{(\ell-r)+}.
\end{equation*}
\end{enumerate}
\end{lem}
\begin{proof}
  First, assume $r >0$, so $[\n]_\phi$ is a representative of $\b$. In
  the notation of Lemma~\ref{cores}, $[\n]_\phi - \pi_\fs^* \circ
  \rho_\fs ([\n]_\phi) + \hfg^\vee_{x, (\ell-r)+} \in \ker
  (\brho_\fs)$.  Part~\eqref{cores3} of the Lemma implies that there
  exists $X \in \hfg_{x, \ell}$ such that $[\n]_\phi + \ad^*(X)
  ([\n]_\phi) - \pi_\fs^* \circ \rho_\fs([\n]_\phi) \in \hfg^\vee_{x,
    (\ell-r)+}$.  Set $p = \exp(X) $.  By
  Lemma~\ref{actlem}\eqref{act4}, $[\n]_{p \phi} - \tx \frac{dz}{z}
  \in \Ad^*(p) \left( [\n]_\phi - \tx \ddz \right) + \hfg^\vee_{x,
    \ell}$.  Since $\Ad^*(p) ([\n]_\phi - \tx \ddz) \in [\n]_\phi -
  \tx \ddz + \ad^* (X)([\n]_\phi)+ \hfg^\vee_{x, (\ell-r)+}$ and
  $\hfg^\vee_{x, \ell}\subset \hfg^\vee_{x, (\ell-r) +}$, it follows
  from the observations above that $[\n]_{p \phi} - \tx \ddz \in
  \pi^*_\fs \circ \rho_\fs ([\n]_\phi) - \tx \ddz + \hfg^\vee_{x,
    (\ell-r)+}.$ This proves the first statement in the irregular
  singular case.
  
  Writing $q = \exp(Y)$ for $Y\in \hfg_{x, \ell}$, we obtain
\begin{multline*}
[\n]_{q \phi} - \tx \ddz \in \Ad^*(q) ([\n]_\phi - \tx \ddz) + \hfg^\vee_{x, (\ell-r)+}
\\ = [\n]_\phi  + \ad^*(Y) ([\n]_\phi)- \tx \ddz + \hfg^\vee_{x, (\ell-r)+}.
\end{multline*}
Since $\ad^*(Y) ([\n]_\phi) + \hfg_{x, (\ell-r)+} \in \ker
(\brho_\fs)$ by Lemma~\ref{cores}\eqref{cores3}, the second statement
follows.

Now, assume that $r = 0$ and $S=\hT$, and write $[\n]_\phi - \tx \ddz
\in \pi_{\hft}^* \rho_{\hft} ([\n]_\phi) - \tx \ddz + \sum_{\psi \in
  \Phi}Y_\psi \ddz + \hfg^\vee_{x, \ell+}$, where $Y_\psi \in
\hfu_{\psi}\cap\hfg_{x, \ell}.$ Choose $\a \in \Phi$ such that $\ell -
\a(\tx) \in \Z_{> -\a(\tx)}$; this condition is necessarily satisfied
by $\a$ if $Y_\a \ddz \in \hfg^\vee_{x, \ell} \backslash \hfg^\vee_{x,
  \ell+}$.  Recall that the graded representative $\tbo \in [\n] _\phi
- \tx \ddz + \hfg^\vee_{x, 0+}$ satisfies $b = \Res(\tbo) \in \ft$.
The nonresonance condition ensures that $\ell + \a(b) \ne 0$.  Define
$X_\a = \frac{1}{\ell + \a(b)} Y_\a$.  By
Lemma~\ref{actlem}\eqref{act2},
\begin{equation*}
\begin{aligned}
[\n]_{\exp(X_\a)\phi} - \tx \ddz & \in \Ad^*(\exp(X_\a)) ([\n]_\phi - \tx \ddz) - \ell X_\a \ddz + \hfg^\vee_{x, \ell+} \\
& = [\n]_\phi - \tx \ddz - (\a(b) + \ell)X_\a \ddz + \hfg^\vee_{x, \ell+} \\
& = \pi_{\hft}^* \rho_{\hft} ([\n]_\phi) - \tx \ddz + 
\sum_{\substack{\psi \in \Phi \\ \psi \ne \a}} Y_\a \ddz + \hfg^\vee _{x, \ell+}
\end{aligned}
\end{equation*}
Repeating this process, we can kill off all of the off-diagonal terms
$Y_\a \ddz$. This completes the proof.
\end{proof}

\begin{lem}\label{tdiff}  Suppose that $x \in \Ao$ is compatible with $S$.
  Let $Z \ddz \in \pi_\fs^*(\fs_\ell^\vee)$ for $\ell > 0$; here,
  $Z\in \fs_\ell$.  Then, there exists $s \in S_\ell$ such that $(ds)
  s^{-1} \in (Z - \ad(\tx) (\frac{1}{\ell} Z)) \ddz + \hfg^\vee_{x,
    \ell+}$.  In particular, $\pi_\fs^*\circ \rho_\fs ((ds) s^{-1})
  \in Z \ddz + \hfg^\vee_{x, \ell +}$.
\end{lem}
\begin{proof}
  Take $s = \exp(\frac{1}{\ell} Z)$.  Then, $(ds) s^{-1} =
  \frac{1}{\ell} \tau (Z) \frac{dz}{z}$.  However, $\frac{1}{\ell}\tau
  (Z) \in - \ad (\tx)(\frac{1}{\ell}Z) + Z + \hfg_{x,\ell+} $.
  Consider $\ad^*(\tx) (\frac{1}{\ell}Z \ddz)$.  If $X \in \fs$, then
  $\ad^*(\tx) (\frac{1}{\ell}Z \ddz) \left( X \right) = -\ad^*(X)
  (\frac{1}{\ell}Z \ddz) \left(\tx \right) = 0$.  It follows that
  $\rho_\fs (\ad^*(\tx) (\frac{1}{\ell} Z \ddz) ) = 0$ by definition
  of the restriction map.  Therefore, $\rho_\fs ((ds)s^{-1}) \in \rho_\fs(Z
  \frac{dz}{z}) + \fs^\vee_{\ell+}$.
\end{proof}

\begin{proof}[Proof of Theorem~\ref{diagthm}]

  By Theorem~\ref{compthm}, we may assume that $S$ is graded
  compatible with $x$.  
  Suppose that $\phi'$ is a trivialization for which $[\n]_{\phi'} \in
  \pi_\fs^*(\fs^\vee) + \hfg^\vee_{x, \ell-r}$ with $\ell >0$.
  Writing $\rho_\fs([\n]_\phi) = \sum_{i\ge -r} A_i$ with $A_i \in
  \fs^\vee (i),$ we have $[\n]_{\phi'}=\sum_{-r\le i<\ell-r} \tA_i
  +\hfg^\vee_{x,\ell-r}$.  We suppose further that if $\ell - r > 0$,
  then $A_i=0$ for $0 < i < \ell -r$.  We now show that we can
  construct $p_\ell$, in $\hG_{x,\ell-r}$ if $\ell-r>0$ and in
  $\hG_{x,\ell}$ if not, such that $[\n]_{p_\ell\phi'}=\sum_{-r\le
    i<\ell-r} \tA_i +\pi^*_\fs(\fs(\ell-r))+\hfg^\vee_{x,(\ell-r)+}$.
  If we can do this, then applying this process recursively, starting
  with $\phi$ and the smallest $\ell'$ for which
  $\hfg^\vee_{x,(\ell'-r)+}\ne\hfg^\vee_{x,-r+}$, we obtain a
  well-defined $p=\prod_{\ell=\ell'}^\infty p_\ell$ satisfying
  part~\eqref{diag1}.  Moreover, we can set $A=\sum_{i = -r}^0 A_i$.

  By Lemma~\ref{redstep}, there exists $q \in \hG_{x, \ell}$ such that
  $[\n]_{q \phi'} - \tx \ddz \in \pi_\fs^*\circ\rho_\fs ([\n]_{\phi'}) -
  \tx \ddz+ \hfg^\vee_{x, (\ell-r)+}$.  If $\ell-r\le 0$, set
  $p_\ell=q$.  On the other hand, if $\ell - r > 0$, Lemma~\ref{tdiff}
  shows that there exists $s \in S_{\ell -r}$ such that $\rho_\fs
  ((ds) s^{-1}) \in A_{\ell - r} + \fs^\vee_{(\ell -r )+}$.  In
  particular, $ \rho_\fs ([\n]_{s q \phi}) \in \sum_{i = -r}^0 A_i +
  \fs^\vee_{ (\ell -r)+}$.  Applying Lemma~\ref{redstep} once more,
  we obtain $q' \in \hG_{x, \ell}$ such that $ [\n]_{q' s q \phi}
  - \tx \ddz \in \sum_{i = -r}^0 \tA_i - \tx \ddz + \hfg^\vee_{x,
    (\ell-r)+}$.  The desired change of trivialization is given by $p_\ell = q' s q$.  

  Part~\eqref{diag2} now follows from the observation that whenever $X \in
  \hfg^\vee_{x, 0+}$, the preceding algorithm produces an element $p
  \in \hG_{x, 0+}$ such that $p \cdot (\tA + X) = \tA$.

    It remains to prove
  part~\eqref{diag3}.  Recall that Lemma~\ref{actlem}\eqref{act3}
  states that $p \cdot \tA^1- \tx \ddz \in \Ad^*(p) (\tA^1 - \tx \ddz)
  + \hfg_{x, 0+}^\vee$.  It follows that $\Ad^*(p) (\tA^1 - \tx \ddz)
  \in \tA^2- \tx \ddz + \hfg_{x, 0+}^\vee$.  Since $\tA^1$ and $\tA^2$
  determine the same regular stratum, $\Ad^*(p) (\tA^1 - \tx \ddz) \in
  \tA^1 - \tx \ddz + \hfg^\vee_{x, -r+}$.

  We first consider the case $r > 0$.  By
  Proposition~\ref{leadingterm}, there exists $q \in \hG_{x, r}$ such
  that $\Ad^*(q p) (\tA^1) \in \tA^2 +\pi_\fs^*(\fs^\vee_{0+})$, so $n
  = q p \in N(S)\cap\hG_{x}$.  Write $\tA^1 \in \tA_{-r} +
  \pi_\fs^*(\fs^\vee_{-r+})$ where $\tA_{-r} \in
  \pi_\fs^*(\fs^\vee(-r))$; note that $\tA^2 \in \tA_{-r} +
  \pi_\fs^*(\fs_{-r+}^\vee)$ by assumption.  Fix a $w$-diagonalizer
  $g$ for $S$ satisfying Theorem~\ref{gx}, and write $n' = g^{-1} n g
  \in N(E)$.  A direct calculation, using the fact that $\tau(n')
  (n')^{-1} \in \ft(E)$, shows that $\tau(n) n^{-1} \in \tau(g) g^{-1}
  -\Ad(n)(\tau(g) g^{-1})+ \fs = - \tx + \Ad(n) (\tx)+ \fs$.
  Therefore,
\begin{multline*}
  \tau(\Ad^*(n) (\tA_{-r})) + \ad^*(\tx) (\Ad^*(n) (\tA_{-r})) \\
  = \Ad^*(n) (\tau(\tA_{-r})) + \ad^*(\tau(n) n^{-1}) (\Ad^*(n)
  (\tA_{-r}))
  + \ad^*(\tx) (\Ad^*(n) (\tA_{-r}))\\
  = \Ad^*(n) (\tau(\tA_{-r}) + \ad^*(\tx) (\tA_{-r})) = -r
  \Ad^*(n) (\tA_{-r}).
\end{multline*}
We deduce that $\Ad^*(n) (\tA_{-r} ) = \tA_{-r}$.  Since $\tA_{-r}$ is
regular semisimple, $n\in S\cap\hG_x=S_0$.  Using the facts
that $(dn) n^{-1} \in \Ad^*(n) (\tx) - \tx + \fs^\vee_{0+}$ and
$\rho_s$ commutes with $\Ad^*(n)$, we see that $\rho_\fs (n \cdot
\tA^1) \in \rho_\fs(\tA^1) + \fs^\vee_{0+}$.  We now apply
Lemma~\ref{redstep}~\eqref{redstep2} to obtain
\begin{equation*}
  \rho_\fs(\tA^2) \in \rho_\fs(q^{-1}n \cdot \tA^1)+ \fs^\vee_{0+}= \rho_\fs (n \cdot \tA^1) +
  \fs^\vee_{0+} = \rho_\fs(\tA^1) + \fs^\vee_{0+}.
\end{equation*}
This proves part~\eqref{diag3} when $r >0$.

Finally, assume that $r = 0$.  Since $\tA^1$ and $\tA^2$ both
determine the same stratum, $\tA^1 - \tx \ddz + \hfg^\vee_{x+} = \tA^2
- \tx \ddz + \hfg^\vee_{x+}$.  We thus have
$\tA^1-\tA^2\in\pi^*_{\hft}(\hft_0^\vee)\cap\hfg^\vee_{x,
  0+}=\pi^*_{\hft}(\hft^\vee_{0+})$.

\end{proof}
\subsection{Flat $G$-bundles  and orbits of formal types}

Let $S$ be a maximal torus of type $\g$, which for the time being is
not assumed to be regular.  Let $W_S = N(S) /S$ and $\hW_S=N(S) / S_0$
be the relative Weyl group and the relative affine Weyl group
associated to $S$.  Note that $\hW_S \cong W_S \ltimes S / S_0$.  The
group $\hW_S$ consists of the Galois fixed points of $N(S(E)) / S_0(E)
\cong \hW$.

\begin{prop}
  The group $W_S$ is isomorphic to a subgroup of the centralizer in
  $W$ of a representative $w$ for $\g$.
\end{prop}
\begin{proof}
  In order to prove this, let $n \in N(S)$ and choose a regular
  element $s \in \fs$. Choose a $w$-diagonalizer $g$ for $S$. Then,
  $s=\Ad(g)(t)$ for some $t \in \ft(E)$ that satisfies $\s (t) =
  w^{-1} t$.  Write $n' = g^{-1} n g$.  It is clear that $n' \in
  N(E)$.  Moreover, since $N(E)=N\cdot T(E)$, one sees that $\s (n' t
  (n')^{-1}) = n' \s(t) (n')^{-1}$.  Thus, if $u \in W$ is the image
  of $n'$, $w^{-1} u t = u w^{-1} t$.  Since $t$ is regular, $w^{-1} u
  = u w^{-1}$.  It follows that $n\mapsto u$ defines a monomorphism
  from $N(S)/S$ into the centralizer of $w$ in $W$.
\end{proof}

We now show that there is an action $\hvr$ of $N(S)$ on $\fs^\vee /
\fs^\vee_{0+}$ by affine transformations given by the formula
\begin{equation*}
\hvr(n) (X + \fs^\vee_{0+}) = \Ad^*(n) (X) - \rho_\fs( (dn) n^{-1}) + \fs^\vee_{0+}.
\end{equation*}

\begin{prop}\label{vrho}
  The map $\hvr : N(S) \to \Aff(\fs^\vee / \fs^\vee_{0+})$ defined
  above is a group homomorpism.  The kernel of $\hvr$ contains $S_0$,
  so $\hvr$ induces a group action $\vr$ of $\hW_S$ on $\fs^\vee /
  \fs^\vee_{0+}$.  Finally, $\fs^\vee_{-r} / \fs^\vee_{0+} \subset
  \fs^\vee/ \fs^\vee_{0+}$ is a finite dimensional submodule for all
  $r\ge0$, and the quotient action on $\fs^\vee / \fs^\vee_0$ is comes
  from the coadjoint action of $N(S)$.
\end{prop}
\begin{proof}
  Without loss of generality, assume that $\fs$ is graded compatible
  with $x \in \Ao$.  Choose a $w$-diagonalizer $g \in G(E)$ such that
  $\Ad(g) (\ft(E)) = \fs(E)$ and $(\tau g) g^{-1} \in \tx + \fs(0)$ as
  in Proposition~\ref{vandermonde}.

  Suppose that $n_1, n_2 \in N(S)$.  Then, $\rho_\fs (\Ad^*(n_1) ((d
  n_2) n_2^{-1}) = \Ad^*(n_1) \rho_\fs ((d n_2) n_2^{-1})$.  It
  follows that
\begin{equation*}
\Ad^*(n_1 n_2) (X) - \rho_\fs((d (n_1 n_2)) (n_1 n_2)^{-1}) = 
\Ad^*(n_1) \left[ \Ad^* (n_2) (X) - \rho_\fs ((d n_2 ) n_2^{-1}) \right] 
- d (n_1) n_1^{-1},
\end{equation*}
and thus $\hvr$ is a homomorphism.  Now, take $s \in S_0$.  Write $s =
t s'$, where $t \in S_0 \cap T$ and $s' \in S_{0+}$.  By
Lemma~\ref{bnorm}, $-(dt) t^{-1} \in \ft^\vee_{0+}$.  Since $s' = \exp
(X)$ for some $X \in \fs_{0+}$, $\rho_\fs (-(ds' ) (s')^{-1}) =
\rho_\fs(- \tau (X) \ddz) \in \fs^\vee_{0+}$.  It follows that $S_0$
lies in the kernel of $\hvr$.

We now prove the second half of the proposition.  Let $n \in N(S)$.
There exists $n' \in N(E)$ such that $g (n')g^{-1} = n$.  Then,
\begin{equation*}
\begin{aligned}
  (dn) n^{-1} &= \Ad^*(g n') (-g^{-1} dg) + \Ad^*(g) (dn' (n')^{-1}) + (dg) g^{-1}\\
  & =\Ad^*(g)(dn' (n')^{-1}) - \Ad^*(n) ((dg) g^{-1}) + (dg) g^{-1}.
\end{aligned}
\end{equation*}
Since $(dg) g^{-1} \in \hfg^\vee_x (0)$, it follows that $\rho_s((d g)
g^{-1} ) \subset \fs^\vee(0)$.  Moreover, $\rho_s (\Ad^*(n) ((d g)
g^{-1})) = \Ad^*(n) (\rho_s((d g) g^{-1}))$, so the restriction of the
second two terms in the expression above lie in $\fs^\vee(0)$.  By
Lemma~\ref{bnorm}, $(d n') (n')^{-1} \in \left( \Ad(n') (\tx) -
  \widetilde{n' x} \right) \ddz + \ft_{0+}(E) \ddz \subset \ft_0(E)
\ddz$.  Thus, $\Ad^*(g) ((dn') (n')^{-1}) \in
\pi^*_\fs(\fs^\vee(E)_{0}) \cap \hfg^\vee = \pi^*_\fs(\fs^\vee_0)$. We
conclude that $\rho_\fs((dn) n^{-1}) \in \fs^\vee_0$.  This proves
that $\hvr(n) (X+\fs^\vee_{0+}) + \fs^\vee_0 = \Ad^*(n) (X) +
\fs^\vee_0$ whenever $X \in \fs^\vee$.

Finally, since the action of the Weyl group preserves the Moy-Prasad
filtration on a split torus, it follows that $\Ad^*(n)
(\fs^\vee(E)_{-r}) = \fs^\vee(E)_{-r}$.  Thus, $\Ad^*(n)
(\fs^\vee_{-r}) = \fs^\vee_{-r}$.  Since $(dn) n^{-1} \in \fs^\vee_0$,
it is now clear that $\hvr(n) (\fs^\vee_{-r}) \subset \hvr(n)
(\fs^\vee_{-r})$ whenever $ r \ge 0$.
\end{proof}

From now on, we reimpose the conditions on $S$ from the previous
section.  In particular, $\g$ is a regular conjugacy class.  First, we
show that the action $\vr$ restricts to give an action on the space of
$S$-formal types of depth $r$.

\begin{prop}
The subspace $\AT(S, r) \subset \fs_{-r}^\vee / \fs_{0+}^\vee$ is stable under the
action of $\hW_S$.
\end{prop}
\begin{proof}
  By Proposition~\ref{vrho}, $\fs^\vee_{-r} / \fs^\vee_{0+}$ is closed
  under the action of $\hW_S$.  Assume without loss of generality that
  $x \in \Ao$ is graded compatible with $\fs$, and write $X_{-r} \in
  \fs^\vee(-r)$ for the leading term of $X \in \AT(S,r)$.  Suppose
  that $\hw \in \hW$ with representative $n\in N(S)$.  The same
  proposition shows that $\vr(\hw) (X) \in \Ad^*(n) (X) + \fs^\vee_0$.

  First, we assume that $r > 0$.  Since $\Ad^*(n) (\fs(r)) = \fs(r)$,
  it follows that $\vr(\hw)(X)_{-r} = \Ad^*(n) (X_{-r})$.  Thus, $Z^0
  (\vr(\hw)(X)_{-r}) = \Ad^*(n) (Z^0(X_{-r})) = S$.  By
  Proposition~\ref{leadingterm}, $\vr(\hw) (X)$ determines a regular
  stratum.  Therefore, $\vr(\hw) (X) \in \AT(S,r)$.

  Now, we consider the case $r = 0$ and $S=\hT$.  Choose a
  representative $n \in N$ for $\hw$.  The proof of
  Proposition~\ref{ftlem} shows that $X\in\hft^\vee_{0} /
  \hft^\vee_{0+}$ is in $\AT(\hT,0)$ if and only if $\a(\Res(\tX_0))
  \notin \Z$ for every root $\a$. It is obvious that if $X$ satisfies
  this condition, then the same holds for $\Ad^*(n)(X)$.  Also,
  Lemma~\ref{bnorm} shows that $(dn) n^{-1} \in (-\widetilde{n
    x}+\hft_{0+}) \ddz$, so that $\Res((dn) n^{-1}) \in \ft(\Z) \ddz$.
  It follows immediately that $\a(\vr(\hw)(\Res(\tX_0)))\notin\Z$ for
  all $\a$, i.e., $\vr(\hw)(X)\in\AT(\hT,0)$.
\end{proof}

We now show that isomorphisms classes of flat $G$-bundles in
$\cC(S,r)$ can be identified with orbits in $\AT(S,r)$.

\begin{lem}\label{isolem}
Suppose that $x$ is compatible with $\fs$ and 
 $X \in \pi_\fs^* (\fs_{-r}^\vee) + \hfg^\vee_{x}$.  If $g \in
 \hG_{x,r}$ with $r>0$,
 $\rho_\fs (\Ad^*(g) (X) - (dg) g^{-1} ) \in \rho_\fs(X) + \fs^\vee_{0+}$.
\end{lem}
\begin{proof}

  Lemma~\ref{actlem}\eqref{act3} implies that $\Ad^*(g) (X) - (dg)
  g^{-1} \in \Ad^*(g) (X) -\Ad^*(g) (\tx \ddz) + \tx \ddz +
  \hfg^\vee_{x+}.$ Since $g \in \hG_{x+}$, it is clear that $-\Ad^*(g)
  (\tx \ddz) + \tx \ddz \in \hfg^\vee_{x+}$ It thus suffices to show
  that $\rho_\fs (\Ad^*(g) (X)) \in \rho_\fs(X) + \fs^\vee_{0+}$.
  Write $g = \exp (Y)$, for $Y \in \hfg_{x, r}$.  Then, $\Ad^*(g) (X)
  \in X + \ad^*(Y) (X) + \hfg^\vee_{x+}$.  By Remark~\ref{tamermk},
  $\rho_\fs(X + \ad^*(Y) (X) + \hfg^\vee_{x+}) \in \rho_\fs(X) +
  \fs^\vee_{0+}$.
\end{proof}

\begin{thm}\label{isothm}
   There is a
  bijection between the set of isomorphism classes of formal flat
  $G$-bundles that contain an $S$-regular stratum of depth $r$
  and the set $\AT(S, r) / \hW_S$.
\end{thm}
\begin{proof}
  Theorem~\ref{diagthm} shows that whenever $(\sG, \n)$ contains a
  regular stratum $(x, r, \b)$, there is a trivialization $\phi$ such
  that $[\n]_\phi \in \pi_\fs^*(\fs_{-r}^\vee)$.  Moreover, the
  isomorphism class of $(\sG, \n)$ depends only on the restriction of
  $[\n]_{\phi}$ to $\fs^\vee_{-r} / \fs^\vee_{0+}$, giving rise to an
  element $A\in\AT(S,r)$.  One further sees that every formal type in
  the $\hW_S$-orbit of $A$ can be obtained from $(\sG, \n)$ by
  changing the trivialization by elements of $\hN$.  It remains to
  show the converse.

  With $\phi$ a trivialization for $\sG$ as above, take $g \in \hG$
  such that $(\sG, \n)$ contains the regular strata $(y, r, \b')$ with
  respect to the trivialization $g \phi$.  Note that the depths are
  necessarily the same by Theorem~\ref{MP}.  We may assume that
  $[\n]_{g \phi} \in \pi_\fs^*(\fs_{-r}^\vee)$ by
  Theorem~\ref{diagthm}\eqref{diag1}, so that $\b'=\rho_\fs([\n]_{g
    \phi})^y$. If $r>0$, Proposition~\ref{ftlem} implies that
  $(x,r,\rho_\fs([\n]_{g \phi})^x)$ is a regular stratum contained in
  $(\sG, \n)$ with respect to the trivialization $g\phi$.  On the
  other hand, if $r=0$, the same proposition shows that for any $x'$
  in a nonempty open subset of the fundamental alcove, $(\sG, \n)$
  contains the regular strata $(x',0,\rho_{\hft}([\n]_{\phi})^{x'})$
  and $(x',0,\rho_{\hft}([\n]_{g\phi})^{x'})$ with respect to $\phi$
  and $g\phi$ respectively.  Thus, we may assume without loss of
  generality that $x=y$ and further, that $\hG_x$ is the standard
  Iwahori subgroup $I$ when $r=0$.

  First, assume that $r > 0$.  Write $g = p_1 n p_2$ using the affine
  Bruhat decomposition, where $p_1, p_2 \in \hG_x$ and $n \in \hN$.
  By Lemma~\ref{actlem}\eqref{act3}, $[\n]_{p_2 \phi} - \tx \ddz \in
  \Ad^*(p_2) ([\n]_\phi - \tx \ddz) + \hfg^\vee_{x+}$.  Applying
  $\Ad^*(p_2^{-1})$ to both sides,
\begin{equation*}
[\n]_\phi - p_2^{-1} (d p_2) \in [\n]_\phi - \tx \ddz + \Ad^*(p_2^{-1}) (\tx \ddz) + 
\hfg^\vee_{x+} \subset [\n]_\phi + \hfg^\vee_x.
\end{equation*}
By Lemma~\ref{leadingterm}\eqref{leadingterm2}, we see that there exists
$q_2 \in \hG_{x, r}$ such that
\begin{equation}\label{ith1}
Z_1 := \Ad^*(q_2^{-1}) \left( [\n]_{\phi} -p_2^{-1} d p_2  \right) \in
[\n]_\phi +\pi_\fs^*( \fs_0^\vee).
\end{equation}
Using the fact that $(dn)n^{-1}\in\hfg^\vee_x$, a similar argument
shows that there exists $q_1 \in \hG_{x, r}$ such that
\begin{equation} \label{ith2}
Z_2 :=\Ad^*(q_1) \left([\n]_{g \phi} + (d p_1) p_1^{-1} + \Ad^*(p_1) ((d n) n^{-1}) \right)
\in
[\n]_{g \phi} + \pi_\fs^*(\fs_0^\vee).
\end{equation}

Note that $Z_1$ and $Z_2$ are regular semisimple elements of
$\pi_\fs^*(\fs^\vee)$ since they both determine regular strata.
Setting $h = q_1 g q_2$, the following calculation shows that
$\Ad^*(h) (Z_1) = Z_2$ and thus $h \in N(S)$:
\begin{equation*}
\begin{aligned}
\Ad^*(h) (Z_1) & = 
\Ad^*(q_1 p_1 n) \left( [\n]_{p_2 \phi} \right)  \\
& = \Ad^*(q_1 p_1) \left( [\n]_{n p_2 \phi}  + (d n) n^{-1} \right)  \\
& = \Ad^*(q_1) \left( [\n]_{g \phi}  + \Ad^*(p_1) \left((d n) n^{-1}\right) + (d p_1) p_1^{-1} \right) = Z_2. 
\end{aligned}
\end{equation*}

Now, let $X = \rho_\fs([\n]_\phi) + \fs_{0+}^\vee$ and $Y =
\rho_\fs([\n]_{g \phi}) + \fs_{0+}^\vee$, so that $X, Y \in \AT(S,
r)$.  We will show that $Y = \hvr(h) (X)$.  Applying
Lemma~\ref{redstep}\eqref{redstep2}, we see that $\rho_\fs([\n]_{q_1 g
  \phi }) \in \rho_\fs([\n]_{g \phi})+ \fs_{0+}^\vee$ and
$\rho_\fs([\n]_{q_2^{-1} \phi }) \in \rho_\fs([\n]_{\phi})+
\fs_{0+}^\vee$.  Since $\rho_\fs$ is an $N(S)$-map by
Lemma~\ref{cores}\eqref{cores5}, it follows that
\begin{multline*}
  \rho_\fs ([\n]_{h q_2^{-1} \phi}) = \rho_\fs (\Ad^*(h)
  ([\n]_{q_2^{-1}\phi})) - (dh)h^{-1}) \\ = \Ad^*(h) (\rho_\fs
  ([\n]_{q_2^{-1} \phi})) - \rho_\fs( (dh) h^{-1})\in \rho_\fs
  ([\n]_{h \phi})+\fs_{0+}^\vee.
\end{multline*}
Finally, 
\begin{multline*}
  Y  = \rho_\fs([\n]_{g \phi}) + \fs_{0+}^\vee =  \rho_\fs([\n]_{q_1 g \phi}) + \fs^\vee_{0+}\\
  = \rho_\fs([\n]_{h q_2^{-1} \phi}) + \fs^\vee_{0+} =
  \rho_\fs([\n]_{h \phi}) + \fs^\vee_{0+} = \hvr(h) (X).
\end{multline*}

Now, take $r=0$ with $\hG_x=I$.  Again, write $g = p_1 n p_2$ with
$p_i \in I$ and $n \in \hN$.  Furthermore, since $I = T \ltimes I_+$,
we may assume that $p_1, p_2 \in \hG_{x+}=I_+$.

Since $p_i \in I_+$, $\rho_{\hft} ([\n]_{p_2 \phi}) \in \rho_{\hft}
([\n]_{\phi}) + \hft^\vee_{0+}$ and $\rho_{\hft} ([\n]_{p_1^{-1} g
  \phi}) \in \rho_{\hft} ([\n]_{g \phi}) + \hft^\vee_{0+}$ by
Lemma~\ref{actlem}\eqref{act4}.  Finally,
\begin{multline*}
  \rho_{\hft} ([\n]_{g \phi})\in \rho_{\hft} ([\n]_{n p_2 \phi}) +
  \hft^\vee_{0+} = \Ad^*(n) (\rho_{\hft} ([\n]_{p_2 \phi})) - (dn)
  n^{-1} + \hft^\vee_{0+}\\ = \Ad^*(n) (\rho_{\hft} ([\n]_\phi)) -
  (dn) n^{-1} + \hft^\vee_{0+}= \hvr(n) (\rho_{\hft} ([\n]_\phi) +
  \hft^\vee_{0+}).\qedhere
\end{multline*}
\end{proof}

\begin{cor} The moduli space of $\cC(S,r)$ is $\AT(S,r)/\hW_S$.  Moreover,
  the deframing functor $\cCfr_x(S,r)\to\cC(S,r)$ corresponds to the
  quotient map $\AT(S,r)\to\AT(S,r)/\hW_S$ on moduli spaces when $r>0$
  or $r=0$ and $H_x=G$.
\end{cor}

Thus, the category $\cCfr_x(S,r)$ (with $H_x=G$ if $r=0$) may be
viewed as a ``resolution'' of $\cC(S,r)$.

\end{document}